\documentclass[a4paper]{amsart}

\usepackage[margin=1.5in]{geometry}
\usepackage{amsmath}%
\usepackage{amsfonts}%
\usepackage{amsthm}
\usepackage{amssymb}%
\usepackage{graphicx}
\usepackage{faktor}
\usepackage{xfrac}
\usepackage{xcolor}
\usepackage{verbatim}
\usepackage{mathtools}
\usepackage{enumerate}
\usepackage{dcolumn}
\usepackage{booktabs}
\usepackage{ stmaryrd }
\usepackage{cite}
\usepackage{tikz}
\usetikzlibrary{arrows}
\usepackage{tikz-cd}
\usepackage{hyperref}
\usepackage{array}   
\newcolumntype{L}{>{$}l<{$}} 
\newcolumntype{C}{>{$}c<{$}} 
\usepackage{todonotes}
\usepackage{amsaddr} 

\newtheorem{theorem}{Theorem}[section]
\theoremstyle{plain}
\newtheorem{corollary}[theorem]{Corollary}
\newtheorem{lemma}[theorem]{Lemma}
\newtheorem{proposition}[theorem]{Proposition}

\theoremstyle{definition}
\newtheorem{definition}[theorem]{Definition}
\newtheorem{example}[theorem]{Example}

\numberwithin{equation}{section}

\def\quotient#1#2{%
    \raise1ex\hbox{$#1$}\Big/\lower1ex\hbox{$#2$}%
}

\DeclareMathOperator{\sgn}{sgn} 
\DeclareMathOperator{\Disc}{Disc} 
\DeclareMathOperator{\id}{id}

\newcommand{\uv}[1]{``{#1}"}

\newcommand{\R}{\mathbb{R}}

\newcommand{\Z}{\mathbb{Z}}
\newcommand{\Q}{\mathbb{Q}}

\newcommand{\abs}[1]{\left|{#1}\right|} 

\newcommand{\ii}{\mathrm{i}}

\newcommand{\OK}[1]{\ensuremath{\mathcal{O}_{#1}}} 
\newcommand{\ve}{\varepsilon}
\newcommand{\cjg}[1]{\overline{#1}} 
\newcommand{\chv}[1]{ \sgn(\sigma_1({#1})), \dots,  \sgn(\sigma_r({#1}))} 
\newcommand{\barsgn}[1]{\underline{\sgn}\left(#1\right)} 

\newcommand{\norm}[2]{\mathcal{N}_{#1}\left(#2\right)} 

\newcommand{\OrelISet}[1]{\mathcal{I}^{o}_{#1}} 
\newcommand{\OrelPSet}[1]{\mathcal{P}^{o}_{#1}} 
\newcommand{\OrelCl}[1]{\mathcal{C\!\ell}^{o}_{#1}} 

\newcommand{\ISet}[1]{\mathcal{I}_{#1}} 
\newcommand{\PSet}[1]{\mathcal{P}_{#1}} 
\newcommand{\Cl}[1]{\mathcal{C\!\ell}_{#1}} 

\newcommand{\USet}[1]{\mathcal{U}_{#1}} 
\newcommand{\UPlusSet}[1]{\mathcal{U}^{+}_{#1}} 

\newcommand{\OI}[2]{\left(#1; #2_1, \dots, #2_r \right)} 

\newcommand{\OrelI}[2]{\left(#1; #2_1, \dots, #2_r \right)} 
\newcommand{\OrelIbasis}[2]{\left( \left[#1\right]; \barsgn{\det #2} \right)} 
\newcommand{\OrelP}[1]{\left( \left(#1\right); \barsgn{ #1\cjg{#1} } \right) } 
\newcommand{\OrelPnorm}[2]{\left( \left(#1\right); \barsgn{ \norm{#2}{#1} } \right) } 

\newcommand{\DSet}{\mathcal{D}} 
\newcommand{\QFSet}{\mathcal{Q}_\DSet}

\newcommand{\CSet}{\mathcal{C}_\DSet} 

\newcommand{\Bal}{\mathcal{B}al\big(\OrelCl{L/K}\big)} 
\newcommand{\TripI}{\mathcal{T}rip\big(\OrelCl{L/K}\big)} 
\newcommand{\TripQF}{\mathcal{T}rip\left(\QFSet\right)}

\newcommand{\GL}{\mathrm{GL}_2^+(\OK{K})}
\newcommand{\pqrs}{\begin{psmallmatrix}p&q\\r&s \end{psmallmatrix}}
\newcommand{\prqs}{\begin{psmallmatrix}p&r\\q&s \end{psmallmatrix}}

\newcommand{\Mat}{\mathrm{M}}

\newcommand{\ClImag}[1]{\mathcal{C\!\ell}^{i}_{#1}} 
\newcommand{\Gr}[1]{\mathcal{G}_{#1}}

\newcommand{\cubeid}[1]{ 
\tikzset{every node/.style={fill=white}}
\begin{tikzpicture}[scale=1.1, baseline=(current bounding box.center)]
\draw [white] (-2.4, 0) rectangle (5.2, 2);

\node (a) at (0,2) {$0$};
\node (b) at (2,2) {$1$};
\node (c) at (0,0) {$1$};
\node (d) at (2,0) {\scriptsize{$\Omega+\cjg{\Omega}$}};
\node (e) at (0.8,2.8) {$1$};
\node (f) at (2.8,2.8) {\scriptsize{$\Omega+\cjg{\Omega}$}};
\node (g) at (0.8,0.8) {\scriptsize{$\Omega+\cjg{\Omega}$}};
\node (h) at (2.8,0.8) {};

\draw (d) -- (c);
\draw (c) -- (a);
\draw (a) -- (e);
\draw (b) -- (f);
\draw (c) -- (g);
\draw (d) -- (h);
\draw (e) -- (f);
\draw (f) -- (h);
\draw (h) -- (g);
\draw (g) -- (e);
\draw[line width=10pt,white] (a) -- (b);
\draw (a) -- (b);
\draw[line width=10pt, white] (b) -- (d);
\draw (b) -- (d);

\node (hh) at (3.6,0.8) {\scriptsize{$\left(\Omega+\cjg{\Omega}\right)^2-\Omega\cjg{\Omega}$}};

\node (p) at (5, 0.8) {#1}; 
\end{tikzpicture}
}

\newcommand{\cube}[9]{ 
\begin{tikzpicture}[scale=1.1, baseline=(current bounding box.center)]

\node (a) at (0,2) {$#1$};
\node (b) at (2,2) {$#2$};
\node (c) at (0,0) {$#3$};
\node (d) at (2,0) {$#4$};
\node (e) at (0.8,2.8) {$#5$};
\node (f) at (2.8,2.8) {$#6$};
\node (g) at (0.8,0.8) {$#7$};
\node (h) at (2.8,0.8) {$#8$};

\draw (d) -- (c);
\draw (c) -- (a);
\draw (a) -- (e);
\draw (b) -- (f);
\draw (c) -- (g);
\draw (d) -- (h);
\draw (e) -- (f);
\draw (f) -- (h);
\draw (h) -- (g);
\draw (g) -- (e);
\draw[line width=10pt,white] (a) -- (b);
\draw (a) -- (b);
\draw[line width=10pt, white] (b) -- (d);
\draw (b) -- (d);

\node (p) at (3.2, 0.7) {#9}; 
\end{tikzpicture}}

\newcommand{\cubearb}[2][1]{ 
\begin{tikzpicture}[scale=#1, baseline=(current bounding box.center)]

\node (a) at (0,2) {$a$};
\node (b) at (2,2) {$b$};
\node (c) at (0,0) {$c$};
\node (d) at (2,0) {$d$};
\node (e) at (0.8,2.8) {$e$};
\node (f) at (2.8,2.8) {$f$};
\node (g) at (0.8,0.8) {$g$};
\node (h) at (2.8,0.8) {$h$};

\draw (d) -- (c);
\draw (c) -- (a);
\draw (a) -- (e);
\draw (b) -- (f);
\draw (c) -- (g);
\draw (d) -- (h);
\draw (e) -- (f);
\draw (f) -- (h);
\draw (h) -- (g);
\draw (g) -- (e);
\draw[line width=10pt,white] (a) -- (b);
\draw (a) -- (b);
\draw[line width=10pt, white] (b) -- (d);
\draw (b) -- (d);

\node (p) at (3.2, 0.7) {#2}; 
\end{tikzpicture}}

\newcommand{\cubered}[2][1]{ 
\begin{tikzpicture}[scale=#1, baseline=(current bounding box.center)]

\node (a) at (0,2) {$1$};
\node (b) at (2,2) {$0$};
\node (c) at (0,0) {$0$};
\node (d) at (2,0) {$d$};
\node (e) at (0.8,2.8) {$0$};
\node (f) at (2.8,2.8) {$f$};
\node (g) at (0.8,0.8) {$g$};
\node (h) at (2.8,0.8) {$h$};

\draw (d) -- (c);
\draw (c) -- (a);
\draw (a) -- (e);
\draw (b) -- (f);
\draw (c) -- (g);
\draw (d) -- (h);
\draw (e) -- (f);
\draw (f) -- (h);
\draw (h) -- (g);
\draw (g) -- (e);
\draw[line width=10pt,white] (a) -- (b);
\draw (a) -- (b);
\draw[line width=10pt, white] (b) -- (d);
\draw (b) -- (d);

\node (p) at (3.2, 0.7) {#2}; 
\end{tikzpicture}}

\begin{document}
\title[Composition of Bhargava's Cubes]{Composition of Bhargava's Cubes over Number Fields}
\author{Krist\'{y}na Zemkov\'{a}}
\address{Charles University, Faculty of Mathematics and Physics, Department of Algebra,
Sokolovsk\'{a}~83, 18600 Praha~8, Czech Republic}
\address{Department of Mathematical and Statistical Sciences, University of Alberta, Edmonton T6G~2G1, Canada
}
\email{zemk.kr@gmail.com}%
\date{\today}
\thanks{The author was supported by the Czech Science Foundation GA\v{C}R, grant 21-00420M}

\begin{abstract}
In this paper, the composition of Bhargava's cubes is generalized to the ring of integers of a number field of narrow class number one, excluding the case of totally imaginary number fields.
\end{abstract}
\maketitle

\section{Introduction}
The composition of quadratic forms dates back to the 19th century: Gauss, Dirichlet and Dedekind were concerned with binary quadratic forms with coefficients in $\Z$. In the 20th century, different approaches taken in \cite{Butts-Dulin, Butts-Estes, Kaplansky, Kneser, Towber} led to generalizations of the theory to many other rings. Yet a completely new description of the composition of binary quadratic forms over $\Z$ was given by Bhargava in \cite{BhargavaLawsI}: His novel idea is in using \uv{cubes}, i.e., $2\times2\times2$ boxes of integers. A triple of binary quadratic forms can be attached to each cube, and, on the other hand, each cube corresponds to a certain triple of ideals. Via this correspondence, Bhargava obtained not only the Gauss composition of binary quadratic forms, but also a composition law on the cubes themselves. His work was later followed by \cite{Wood} and \cite{ODorney2016}; in particular, O'Dorney \cite{ODorney2016} generalized the composition of cubes to any Dedekind domain.

In this paper, we try to find a balance between \emph{explicit} and \emph{general}; in particular, we look for such a family of rings, over which cubes themselves carry enough information similarly as in \cite{BhargavaLawsI} (which is not the case in \cite{ODorney2016} and \cite{Wood}). This goal is obtained through the results of \cite{KZforms} where a Dedekind-like correspondence was established: For a number field $K$ of narrow class number with at least one real embedding, there is a bijection between classes of binary quadratic forms over $\OK{K}$ (i.e., over the ring of algebraic integers of $K$), and classes of oriented ideals over a quadratic extension of $K$; see Theorem \ref{Theorem:Bijection}. We take advantage of the explicit description of the correspondence: First, in Theorem \ref{Theorem:QFCompFromCube}, we show that the three quadratic forms arising from one cube compose to the identity quadratic form. Afterwards, we use both of the above-mentioned theorems for the proof of Theorem \ref{Theorem:BijectionCubes}, where we extend Bhargava's result to a correspondence between cubes over $\OK{K}$ and certain triples of ideals. 

Although the final result is a straightforward generalization of Bhargava's result, there are (and must be) some significant differences on the path: First, we use a slightly different method to prove that the three quadratic forms arising from one cube compose to identity (see the first paragraph of Section~\ref{Sec:QFcomposition}). But more importantly, we avoid a horrendous computation by \uv{going around} and applying previous results of the author (see the first paragraph of Subsection~\ref{Subsec:CtoI} and Figure~\ref{Fig:DetailedScheme}). 

In Appendix~\ref{Sec:App}, we comment on the case of totally imaginary number fields which we exclude in this paper and point out a mistake in~\cite{KZforms}.

This paper is based on the second part of the author's Master thesis~\cite{KZthesis}.

\section{Preliminaries}
In this section, we summarize the definitions and results from the paper~\cite{KZforms}.

Throughout the whole article, we fix a number field $K$ of narrow class number one; that is, of class number one and with units of all signs (see \cite[Ch. V, (1.12)]{FrohlichTaylor}). Moreover, we assume that $K$ has exactly $r$ embeddings into $\R$ with $r>0$ (see Appendix~\ref{Sec:App} for an explanation why we need this assumption); we denote these embeddings $\sigma_1, \dots, \sigma_r$. By $\OK{K}$, $\USet{K}$, $\UPlusSet{K}$, respectively, we denote the ring of algebraic integers in $K$, the group of units of this ring (i.e., $u\in\OK{K}$ with $\abs{\norm{K/\Q}{u}}=1$), and the subgroup of totally positive units (i.e., $u\in\USet{K}$ such that $\sigma_i(u)>0$ for all $1\leq i\leq r$).

Furthermore, let us fix a number field $L$ such that $L$ is a quadratic extension of $K$. The Galois group $\mathrm{Gal}(L/K)$ has two elements; we write $\cjg{\alpha}$ for the image of $\alpha\in L$ under the action of the nontrivial element of this Galois group.

\subsection{The ring $\OK{L}$}
It is known that, as $K$ is of narrow class number one, the ring $\OK{L}$ is a free $\OK{K}$-module and $\OK{L}=[1, \Omega]_{\OK{K}}$ for a suitable $\Omega\in\OK{L}$ (see \cite[Cor. p. 388]{narkiewicz2004elementary} and \cite[Prop. 2.24]{MilneANT}); for simplification, we will omit the index and write just $\OK{L}=[1, \Omega]$. Let $x^2+wx+z$ be the minimal polynomial of $\Omega$ over $\OK{K}$ and set $D_\Omega=w^2-4z$; then, without loss of generality,
\begin{equation}\label{Omega}
\Omega=\frac{-w+\sqrt{D_{\Omega}}}{2}, \ \ \ \ \ \ \cjg{\Omega}=\frac{-w-\sqrt{D_{\Omega}}}{2},
\end{equation} 
and consequently $D_{\Omega}=\left(\Omega-\cjg{\Omega}\right)^2$.   

As the discriminant of a quadratic form, the element $D_\Omega$ may not necessarily be square-free, but it is \uv{almost square-free}; in fact, it is fundamental (see \cite[Lemma~2.2]{KZforms}):

\begin{definition}\label{Def:AlmostSquare-free}
An element $d$ of $\OK{K}$ is called \emph{fundamental} if $d$ is a quadratic residue modulo $4$ in $\OK{K}$, and
\begin{itemize}
	\item either $d$ is square-free,
	\item or for every $p\in\OK{K}\backslash\:\USet{K}$ such that $p^2\mid d$ the following holds: $p\mid 2$ and $\frac{d}{p^2}$ is not a quadratic residue modulo $4$ in $\OK{K}$.
\end{itemize}
\end{definition}

Obviously, $L=K\left(\sqrt{D_{\Omega}}\right)$; if $d\in\OK{K}$ is another fundamental element satisfying $L=K(\sqrt{d})$, then $d=u^2D_{\Omega}$ for some $u\in\USet{K}$ (see \cite[Lemma~2.3]{KZforms}).

\subsection{Quadratic forms}\label{Subsec:QFs}
Under a \emph{quadratic form}, we understand a homogeneous polynomial of degree 2 with coefficients in $\OK{K}$, i.e., $Q(x, y)=ax^2+bxy+cy^2$ with $a, b, c \in \OK{K}$. By $\Disc(Q)$ we denote the discriminant of the quadratic form $Q$, i.e., $\Disc(Q)=b^2-4ac$.  A quadratic form $Q(x,y)=ax^2+bxy+cy^2$ is called \emph{primitive} if $\gcd(a,b,c)\in\USet{K}$. We say that \emph{two quadratic forms $Q(x,y)$ and $\widetilde{Q}(x,y)$ are equivalent} if there exist $p, q, r, s\in\OK{K}$ satisfying $ps-qr \in \UPlusSet{K}$ and a totally positive unit $u \in \UPlusSet{K}$ such that $\widetilde{Q}(x,y)=u Q(px+qy, rx+sy)$; we denote this by $Q\sim\widetilde{Q}$. Under this definition, the discriminants of equivalent quadratic forms may differ by a square of a totally positive unit; therefore, we will consider all primitive quadratic forms with discriminants in the set 
\[\DSet=\left\{u^2\left(\Omega-\cjg{\Omega}\right)^2 ~\big|~ u \in \UPlusSet{K}\right\}\]
modulo the equivalence relation defined above:
\[\QFSet=\quotient{\left\{Q(x,y)=ax^2+bxy+cy^2 ~\big|~ a,b,c\in\OK{K},\ \gcd(a,b,c)\in\USet{K}, \ \Disc(Q)\in\DSet \right\}}{\sim}.\] 

\subsection{Ideals and orientation} \label{Subsec:IdealsAndOrientation}
Not only the whole ring $\OK{L}$, but also all the (fractional) $\OK{L}$-ideals have an $\OK{K}$-module basis of the form $[\alpha, \beta]$ for some $\alpha, \beta\in L$. In the following, the word ``ideal'' will generally stand for a fractional ideal while  we will refer to the usual meaning as to the ``integral ideal''. 

If $I=[\alpha,\beta]$ is an ideal in $\OK{L}$, then there exists a $2\times2$ matrix $M$ consisting of elements of $K$ such that
\begin{equation}\label{Eq:MatrixM}
\begin{pmatrix}
		\cjg{\alpha} & \alpha \\
		\cjg{\beta} & \beta \\ 
	\end{pmatrix}
	= M\cdot
	\begin{pmatrix}
		1 & 1 \\
		\cjg{\Omega} & \Omega \\
	\end{pmatrix}.
\end{equation}
Then 
\[\det M = \frac{\cjg{\alpha}\beta-\alpha\cjg{\beta}}{\Omega-\cjg{\Omega}}.\]
It is easy to check that $\det M \in K$, and if $\alpha, \beta \in \OK{L}$, then $\det M \in \OK{K}$. Furthermore, if $[p\alpha+r\beta, q\alpha+s\beta]$ is another $\OK{K}$-module basis of $I$ and $\widetilde{M}$ is the matrix corresponding to this basis, then $\det \widetilde{M}=(ps-qr)\det M$.

For $\alpha\in L$, we define the \emph{relative norm} as $\norm{L/K}{\alpha}=\alpha\overline{\alpha}$. Moreover, if $I$ is an $\OK{L}$-ideal, then the \emph{relative norm of the ideal $I$} is the $\OK{K}$-ideal $\norm{L/K}{I}=\left(\norm{L/K}{\alpha}: \alpha\in I \right)$. Note that this ideal has to be principal (because $\OK{K}$ is a principal ideal domain); it is generated by $\det M$ (see \cite[Th. 1]{Mann}).\footnote{This implies that the definition of the relative norm of an ideal agrees with the one used in \cite{BhargavaLawsI}. It is no longer true if we consider other rings than $\OK{K}$;  in such a case, \cite[Th. 1]{Mann} is not applicable. } Obviously $\norm{L/K}{(\alpha)}=\left(\norm{L/K}{\alpha} \right)$.

For $a \in K$  write $\underline{\sgn}(a)=(\chv{a})$. If $I=[\alpha, \beta]$ is an $\OK{L}$-ideal with the matrix $M$ as in \eqref{Eq:MatrixM}, we define the \emph{oriented $\OK{L}$-ideal determined by $[\alpha, \beta]$}  as 
\[\OrelIbasis{\alpha,\beta}{M}=\left( [\alpha, \beta]; \chv{ \det M} \right).\]
Since for the principal ideal generated by $\gamma\in L$, it holds that $(\gamma)=[\gamma,\gamma\Omega]$, it is easy to check that the oriented ideal determined by the principal ideal $(\gamma)$ is equal to
\[\OrelP{\gamma}=\left((\gamma);\chv{\gamma\cjg{\gamma}} \right).\]
Note that, for any given $\OK{L}$-ideal $I$ and any choice of $\ve_i\in\{\pm1\}$, $1\leq i\leq r$, it is possible to find an $\OK{K}$-module basis $[\alpha,\beta]$ of $I$ (with the corresponding matrix $M$) such that $\sgn\sigma_i(\det M)=\ve_i$ for all $1\leq i \leq r$ (because there exist units of all signs in $K$). We call $\OrelI{I}{\ve}$ an \emph{oriented ideal}, and we set 
\[\begin{array}{rcl}
\OrelISet{L/K}&=&\left\{ \OrelI{I}{\ve} ~\left|~ I \text{ a fractional $\OK{L}$-ideal}, \ve_i\in\left\{\pm1\right\}, i=1,\dots,r  \right\} \right., \vspace{1mm}\\ 
\OrelPSet{L/K}&=&\left\{ \OrelPnorm{\gamma}{L/K} ~\left|~ \gamma \in L  \right\} \right.. \vspace{1mm}\\ 
\end{array}\]
Further we define the \emph{relative oriented class group} of the field extension $L/K$ as
\[\OrelCl{L/K}=\quotient{\OrelISet{L/K}}{\OrelPSet{L/K}},\]
where the multiplication on $\OrelISet{L/K}$ is defined componentwise as 
\[\OI{I}{\ve}\cdot\OI{J}{\delta} = \left(IJ; \ve_1\delta_1, \dots, \ve_r\delta_r \right).\]
The following lemma describes the inverses: 

\begin{lemma}[\!\!{\cite[Prop.~2.8~and~Lemma~2.19]{KZforms}}]\label{Lemma:InverseIdeals}
It holds that
\[
\OrelIbasis{\alpha,\beta}{M} \cdot \OrelIbasis{\frac{\cjg\alpha}{\det M}, \frac{-\cjg\beta}{\det M}}{M}
=\left( [1, \Omega]; +1, \dots, +1 \right).
\]
In particular, when considered as elements of $\OrelCl{L/K}$, the inverse to the oriented ideal $\OrelIbasis{\alpha,\beta}{M}$ is the oriented ideal $\OrelIbasis{\cjg\alpha, -\cjg\beta}{M}$.
\end{lemma}

\subsection{The bijection}

The following theorem describes a bijection between the set $\QFSet$ of equivalence classes of primitive quadratic forms of fixed discriminants and the relative oriented class group $\OrelCl{L/K}$, and consequently provides a group structure on the set $\QFSet$.

\begin{theorem}[\!\!{\cite[Thm.~3.6~and~Cor.~3.7]{KZforms}}]\label{Theorem:Bijection}
Let $K$ be a number field of narrow class number one with at least one real embedding and $D$ a fundamental element of $\OK{K}$. Set $L=K\big(\sqrt{D}\big)$ and $\DSet=\left\{u^2D ~|~ u\in\UPlusSet{K}\right\}$. We have a bijection 
\[
\begin{array}{ccc}
\QFSet
&\stackrel{1:1}{\longleftrightarrow} 
& \OrelCl{L/K} \vspace{3mm}\\
Q(x, y)=ax^2+bxy+cy^2 & \stackrel{\Psi}{\longmapsto} & \left(\left[a, \frac{-b+\sqrt{\Disc(Q)}}{2}\right];\barsgn{a}\right) \vspace{2mm}\\
\frac{\alpha\cjg{\alpha}x^2-(\cjg{\alpha}\beta+\alpha\cjg{\beta})xy+\beta\cjg{\beta}y^2}{\frac{\cjg{\alpha}\beta-\alpha\cjg{\beta}}{\sqrt{D}}} 
&\stackrel{\Phi}{\longmapsfrom}
&\left(\left[\alpha, \beta\right];\barsgn{\frac{\cjg{\alpha}\beta-\alpha\cjg{\beta}}{\sqrt{D}}}\right)   
\end{array}
\]
Under this bijection, $\QFSet$ carries a group structure arising from the multiplication of ideals in $L$. The identity element of this group is represented by the quadratic form $Q_{\id}(x,y)=x^2-(\Omega+\cjg{\Omega})xy+\Omega\cjg{\Omega}y^2$, and the inverse element to $ax^2+bxy+cy^2$ is the quadratic form $ax^2-bxy+cy^2$.
\end{theorem}

\section{Cubes over $\OK{K}$}\label{Sec:Cubes}

Bhargava \cite{BhargavaLawsI} introduced cubes of integers as elements of $\Z^2\otimes_{\Z}\Z^2\otimes_{\Z}\Z^2$; if we denote by $\{v_1, v_2\}$ the standard $\Z$-basis of $\Z^2$, then the cube
\begin{equation}\label{cube_111}
\cube{a_{111}}{a_{121}}{a_{112}}{a_{122}}{a_{211}}{a_{221}}{a_{212}}{a_{222}}{}
\end{equation}
with $a_{ijk}\in\Z$ can be viewed as the expression
$$\sum_{i,j,k=1}^2{a_{ijk}v_i\otimes v_j \otimes v_k},$$ 
which is an element of $\Z^2\otimes\Z^2\otimes\Z^2$. 

In this section, we generalize these to number fields: Consider $\OK{K}^2\otimes\OK{K}^2\otimes\OK{K}^2$, this time taking the tensor product over the ring $\OK{K}$, and represent its elements as cubes with vertices $a_{ijk}\in\OK{K}$. We will often refer to such a cube as in \eqref{cube_111} shortly by $(a_{ijk})$ tacitly assuming $a_{ijk}\in\OK{K}$ for all $i,j,k\in\{1,2\}$.

Considering a cube 
\begin{equation}\label{cube_abc}
\cubearb[1.1]{,}
\end{equation}
 it can be sliced in three different ways, which correspond to three pairs of $2\times2$ matrices:
\begin{equation*}\label{Eq:RiSi}
\begin{aligned}
R_1=
\begin{pmatrix}
a & b\\
c & d
\end{pmatrix}
, \ \ \ &
S_1=
\begin{pmatrix}
e & f\\
g & h
\end{pmatrix}, \\
R_2=
\begin{pmatrix}
a & e\\
c & g
\end{pmatrix}
, \ \ \ &
S_2=
\begin{pmatrix}
b & f\\
d & h
\end{pmatrix},\\
R_3=
\begin{pmatrix}
a & e\\
b & f
\end{pmatrix}
, \ \ \ &
S_3=
\begin{pmatrix}
c & g\\
d & h
\end{pmatrix}.
\end{aligned}
\end{equation*}
To each of these pairs, a binary quadratic form $Q_i(x,y)=-\det\left(R_ix-S_iy \right)$ can be assigned:
\begin{equation}\label{Eq:AttachedQFs}\begin{aligned}
Q_1(x,y)&=(bc-ad)x^2+(ah-bg-cf+de)xy+(fg-eh)y^2,\\
Q_2(x,y)&=(ce-ag)x^2+(ah+bg-cf-de)xy+(df-bh)y^2,\\
Q_3(x,y)&=(be-af)x^2+(ah-bg+cf-de)xy+(dg-ch)y^2.
\end{aligned}\end{equation}
Formally, we will look at the assignment of the triple of quadratic forms \eqref{Eq:AttachedQFs} to the cube \eqref{cube_abc} as a map $\widetilde{\Theta}$:

$$ \cubearb[0.8]{} \stackrel{\widetilde{\Theta}}{\longmapsto} \big(-\det(R_1x-S_1y), -\det(R_2x-S_2y), -\det(R_3x-S_3y) \big).$$

One can compute that all the quadratic forms in \eqref{Eq:AttachedQFs} have the same discriminant, namely
\begin{multline*}
\Disc(Q_i)
= a^2h^2+b^2g^2+c^2f^2+d^2e^2 \\- 2(abgh+acfh+aedh+bdeg+bfcg+cdef)+4(adfg+bceh)
\end{multline*}
for every $i=1,2,3$. Hence, we can define the \emph{discriminant} of a cube $A$ as the discriminant of any of the three assigned quadratic forms; we denote this value by $\Disc(A)$. Furthermore, we say that a cube is \emph{projective} if all the assigned quadratic forms are primitive.

Consider the group $\widetilde{\Gamma}=\Mat_2(\OK{K})\times\Mat_2(\OK{K})\times\Mat_2(\OK{K})$, where $\Mat_2(\OK{K})$ denotes the group of $2\times2$ matrices with entries from $\OK{K}$. This group has a~natural action on cubes: If $\pqrs$ is from the $i$-th copy of $\Mat_2(\OK{K})$, $1\leq i\leq3$, then it acts on the cube \eqref{cube_abc} by replacing $(R_i,S_i)$ by $(pR_i+qS_i, rR_i+sS_i)$. Note that the first copy acts on $(R_2,S_2)$ by column operations, and on $(R_3, S_3)$ by row operations. Hence, analogous to the fact that row and column operations on rectangular matrices commute, the three copies of $\Mat_2(\OK{K})$ in $\widetilde{\Gamma}$ commute with each other. Therefore, we can always decompose the action by $T_1\times T_2\times T_3 \in \widetilde{\Gamma}$ into three subsequent actions:
\begin{equation}\label{Eq:ActionDecomp}
T_1\times T_2\times T_3 = (\id\times \id\times T_3)(\id\times T_2\times \id)(T_1\times \id\times \id). 
\end{equation}
Thus, we will usually restrict our attention to the action with only one nontrivial copy.  

Consider the action by $\pqrs\times\id\times\id$ on the cube \eqref{cube_abc}; the resulting cube is
\begin{equation*}
\cube{pa+qe}{pb+qf}{pc+qg}{pd+qh}{ra+se}{rb+sf}{rc+sg}{rd+sh}{\hspace{7mm}.}
\end{equation*}
Let $(Q_1, Q_2, Q_3)=\widetilde{\Theta}(A)$, and let $\left(Q'_1, Q'_2, Q'_3\right)=\widetilde{\Theta}(A')$ with $A'=\left(\pqrs\times\id\times\id\right)(A)$; it is easy to compute that
\begin{align*}
Q'_1(x,y)	&=Q_1(px-ry, -qx+sy),\\ 
Q'_2(x,y) &=(ps-qr)\cdot Q_2(x,y),\\
Q'_3(x,y) &=(ps-qr)\cdot Q_3(x,y).
\end{align*}
We can see that 
$$\Disc\left(\left(\pqrs\times\id\times\id\right)(A) \right)=(ps-qr)^2\Disc(A). $$

Furthermore, for $t\in\OK{K}$, we understand by $tA$ the cube $A$ with all vertices multiplied by $t$. It is clear that $\Disc(tA)=t^4\Disc(A)$. We can summarize our observations into the following lemma.

\begin{lemma} \label{Lemma:ChangeDisc}
Let $A$ be a cube, $t\in\OK{K}$ and $T_1,T_2,T_3\in\Mat_2(\OK{K})$. Then
$$ \Disc\left(t\left(T_1\times T_2\times T_3\right)(A)\right)
=t^4(\det T_1)^2(\det T_2)^2(\det T_3)^2\Disc(A).$$
\end{lemma}

Let us denote
$$\GL=\left\{\left. T \in\Mat_2(\OK{K}) ~\right|~ \det T\in\UPlusSet{K} \right\}, $$
and consider the subgroup 
$$\Gamma=\GL\times\GL\times\GL$$
 of the group $\widetilde{\Gamma}$. Then $\Gamma$ acts on cubes. It follows from the computations above that, for $u\in\USet{K}$ and $T_1\times T_2\times T_3\in\Gamma$, cubes $A$ and $u(T_1\times T_2\times T_3)(A)$ give rise to equivalent triples of quadratic forms. This justifies the following definition.

\begin{definition}
We say that two cubes $A$ and $A'$ are \emph{equivalent} (which will be denoted by $A\sim A'$) if there exists $u\in\USet{K}$ and $T_1\times T_2\times T_3\in\Gamma$ such that $A'=u(T_1\times T_2\times T_3)(A)$.
\end{definition}

Comparing this definition with the results of Lemma~\ref{Lemma:ChangeDisc}, we see that the discriminants of equivalent cubes can differ by a square of a totally positive unit. This agrees with our previous definition:  Let us recall that we denote by $\DSet$ the set of all possible discriminants, i.e., the set $\left\{u^2\left(\Omega-\cjg{\Omega}\right)^2 ~\left|~ u\in\UPlusSet{K}\right.\right\}$. So if $A$ is a cube such that $\Disc(A)\in\DSet$, then all cubes $A'$ equivalent to $A$ satisfy $\Disc(A')\in\DSet$. We will denote by $\CSet$ the set of equivalence classes of projective cubes of discriminant in $\DSet$, i.e.,
$$\CSet=\quotient{\left\{\left. A\in\OK{K}^2\otimes\OK{K}^2\otimes\OK{K}^2~\right|~ A \text{ is projective, } \Disc(A)\in\DSet \right\}}{\sim} .$$ 
Note that if $(Q_1, Q_2, Q_3)$ and $(Q'_1, Q'_2, Q'_3)$ are the images of two equivalent cubes $A$ and $A'$ by $\widetilde{\Theta}$,
then $Q_i\sim Q_i'$ for $1\leq i\leq 3$. Hence, the map $\widetilde{\Theta}$ restricts to a map
\begin{equation}\label{Eq:ThetaDef}
\Theta: \CSet  \longrightarrow  \QFSet\times\QFSet\times\QFSet.
\end{equation}
To relax the notation, we often omit to write the equivalence classes; under ${A\in\CSet}$ we understand the cube $A$ which is a representative of the class $[A]$ in the set $\CSet$, and ${\Theta(A)=(Q_1, Q_2, Q_3)}$ actually means $\Theta\big([A]\big)=\big([Q_1], [Q_2], [Q_3]\big)$.

A cube $A$ is called \emph{reduced} if
$$A=\cubered[1.1]{}$$
for some $d,f,g,h\in\OK{K}$. The following lemma will help us to simplify some of the proofs.

\begin{lemma} \label{Lemma:CubeReduction}
Every projective cube is equivalent to a reduced cube.
\end{lemma}

\begin{proof}
In the case of cubes over $\Z$, the proof with full details can be found in \cite[Sec.~3.1]{Bouyer}. The proof in the case over $\OK{K}$ is completely analogous, because, under our assumptions on $K$, the ring $\OK{K}$ is a principal ideal domain. Here we only outline the rough idea.
Consider a projective cube with vertices $a, b, c, d, e, f, g, h\in\OK{K}$ as in \eqref{cube_abc}.
It follows from projectivity that $\gcd(a, b, c, d, e, f, g, h)=1$; hence, we can find a~cube equivalent to the original one with $1$ in the place of $a$. This can be used to clear out the vertices $b$, $c$ and $e$.
\end{proof}

We point out one specific cube, and that is the cube
\begin{equation}\label{Eq:cubeid}
\cubeid;
\end{equation}
since this cube is triply symmetric, all the three quadratic forms assigned to this cube are equal to 
\[Q_{\id}(x,y)=x^2-\left(\Omega+\cjg{\Omega}\right)xy+\Omega\cjg{\Omega}y^2,\] 
the representative of the identity element in the group $\QFSet$. Therefore, we will denote the cube in \eqref{Eq:cubeid} by $A_{\id}$.
 
\section{Composition of binary quadratic forms through cubes}\label{Sec:QFcomposition}

In \cite[Sec.~2.2]{BhargavaLawsI}, Bhargava establishes \emph{Cube Law}, which says that the composition of the three binary quadratic forms arising from one cube is the identity quadratic form. He shows that, in the case of projective cubes, this law is equivalent to Gauss composition by using Dirichlet's interpretation. In our case, the coefficients of the quadratic forms lie in $\OK{K}$ instead of $\Z$, and hence we are not allowed to use Dirichlet composition without reproving a generalized version. Instead of following this path, we will use the bijective maps $\Psi$ and $\Phi$ between $\QFSet$ and $\OrelCl{L/K}$ from Theorem~\ref{Theorem:Bijection}, and show that the product of the three obtained ideal classes is a principal ideal class.

\begin{lemma}\label{Lemma:ProductOfIdeals}
Let $A\in\CSet$, and denote $\Theta(A)=(Q_1, Q_2, Q_3)$. Let $\mathfrak{J}_i$ be the image of $Q_i$ under the map $\Psi$, $1\leq i\leq 3$. Then there exists an element $\omega\in L$ such that $\mathfrak{J}_1\mathfrak{J}_2\mathfrak{J}_3=\OrelP{\omega}$.
\end{lemma}

\begin{proof}
Using Lemma~\ref{Lemma:CubeReduction}, we can assume that $A$ is reduced. Hence, the quadratic forms arising from this cube are
\begin{equation}\label{Eq:RedForms}\begin{aligned}
Q_1(x,y)&=-dx^2+hxy+fgy^2,\\
Q_2(x,y)&=-gx^2+hxy+dfy^2,\\
Q_3(x,y)&=-fx^2+hxy+dgy^2,
\end{aligned}\end{equation}
all of them having discriminant $D=h^2+4dfg$. Their images under the map $\Psi$ are the oriented ideals
\begin{equation}\label{Eq:OIdeals}\begin{aligned}
\mathfrak{J}_1&=\left(\left[-d, \frac{-h+\sqrt{D}}{2} \right]; \barsgn{-d} \right),\\
\mathfrak{J}_2&=\left(\left[-g, \frac{-h+\sqrt{D}}{2} \right]; \barsgn{-g} \right),\\
\mathfrak{J}_3&=\left(\left[-f, \frac{-h+\sqrt{D}}{2} \right]; \barsgn{-f} \right).\\
\end{aligned}\end{equation}
Denote by $J$ the product of the three (unoriented) ideals; from the multiplicativity of the relative norm follows $\norm{L/K}{J}=(-dfg)$. Set $\omega=\frac{-h+\sqrt{D}}{2}$. We will show that $J=(\omega)=[\omega, \omega\Omega]_{\OK{K}}$; by \cite[Cor.~to~Th.~1]{Mann}, it is sufficient to prove that $\omega \in J$ (because then necessarily $\omega\Omega\in J$ as well) and that $\norm{L/K}{\omega}=-dfg$. 

The latter is clear since $\norm{L/K}{\omega}=\frac14(h^2-D)$ and $D=h^2+4dfg$.
To show that $\omega \in J$, first note that
\begin{align*}
J
&= \left[-dfg, df\omega, dg\omega, fg\omega, -d\omega^2, -f\omega^2, -g\omega^2, \omega^3\right]_{\OK{K}}\\
&= \left[-dfg, df\omega, dg\omega, fg\omega, dh\omega, fh\omega, gh\omega, h^2\omega\right]_{\OK{K}},
\end{align*}
where we have used the relations
$$\omega^2=dfg-h\omega, \ \ \ \ \ \omega^3=-dfgh+(h^2+dfg)\omega.$$
Set $G=\gcd(df,dg,fg,dh,fh,gh,h^2)$; we want to show that $G$ is a unit, because then necessarily $\omega\in J$. It follows from primitiveness of the quadratic forms in \eqref{Eq:RedForms} that $\gcd(d,f,g,h)=1$; therefore, $G=\gcd(df,dg,fg,h)$. Let $p\in\OK{K}$ be a~prime dividing $G$. Since the quadratic form $Q_1$ is primitive and $p$ divides both $fg$ and $h$, we have that $p \nmid d$. But as $p\mid df$, it has to hold that $p\mid f$. Therefore, $p \mid \gcd(f,h,dg)=1$, a contradiction to $p$ being a prime. Hence, both $p$ and $G$ are units. 
\end{proof}

As a consequence of this lemma, we get a new description of the composition of quadratic forms.

\begin{theorem}\label{Theorem:QFCompFromCube}
Let $A\in\CSet$, and let $Q_1, Q_2, Q_3$ are the three quadratic forms arising from the cube $A$. Then their composition $Q_1Q_2Q_3$ is a representative of the identity element of the group $\QFSet$.
\end{theorem}

\begin{proof}
It follows from Lemma~\ref{Lemma:ProductOfIdeals} that $\Psi(Q_1Q_2Q_3)=\Psi(Q_1)\cdot\Psi(Q_2)\cdot\Psi(Q_3)$ is a~principal ideal. Therefore, by Theorem~\ref{Theorem:Bijection},
$$Q_1Q_2Q_3\sim\Phi\Psi(Q_1Q_2Q_3)=x^2-\left(\Omega+\cjg{\Omega}\right)xy+\Omega\cjg{\Omega}y^2,$$ 
i.e., $Q_1Q_2Q_3$ is a representative of the identity element of the group $\QFSet$.
\end{proof}

\section{Composition of cubes}

In this final section, we use the prepared tools to equip the set $\CSet$ with a group law. 

\subsection{Balanced ideals}\label{Sec:BalancedIdeals}

We devote this subsection to the other side of the desired correspondence: balanced triples of ideals. We will show that there are essentially three equivalent views on such a triple. 

Let $\mathfrak{I}_i=\left(I_i; \barsgn{\det M_i}\right)$, $1\leq i\leq3$, be oriented $\OK{L}$-ideals. We say that the triple $\left(\mathfrak{I}_1, \mathfrak{I}_2, \mathfrak{I}_3\right)$ is \emph{balanced} if $I_1I_2I_3=\OK{L}$ and $\det M_1\det M_2\det M_3\in\UPlusSet{K}$. (Note that this is a direct generalization of the definition in \cite[Subsec. 3.3]{BhargavaLawsI}.)

Let $\left(\mathfrak{I}'_1, \mathfrak{I}'_2, \mathfrak{I}'_3\right)$ be another balanced triple, where $\mathfrak{I}'_i=\left(I'_i; \barsgn{\det M'_i}\right)$ for $1\leq i\leq3$. The two balanced triples $\left(\mathfrak{I}_1, \mathfrak{I}_2, \mathfrak{I}_3\right)$ and $\left(\mathfrak{I}'_1, \mathfrak{I}'_2, \mathfrak{I}'_3\right)$  are \emph{equivalent} if there exist $\kappa_i\in L$, $1\leq i\leq3$, such that $\mathfrak{I}'_i=\kappa_i\mathfrak{I}_i$. Note that the equality
$$\OK{L}=I'_1I'_2I'_3=(\kappa_1I_1)(\kappa_2I_2)(\kappa_3I_3)=(\kappa_1\kappa_2\kappa_3)\OK{L}$$
implies that $\kappa_1\kappa_2\kappa_3\in\USet{L}$. Furthermore, since $\det M'_i=\norm{L/K}{\kappa_i}\det M_i$, we have
$\norm{L/K}{\kappa_1\kappa_2\kappa_3}\in\UPlusSet{K}.$ Also note that $\left(\mathfrak{I}_1, \mathfrak{I}_2, \mathfrak{I}_3\right)\sim \left(\frac{1}{\kappa}\mathfrak{I}_1, \kappa\mathfrak{I}_2, \mathfrak{I}_3\right)$ for any $\kappa\in L\backslash\{0\}$.

We will write $\left[\left(\mathfrak{I}_1, \mathfrak{I}_2, \mathfrak{I}_3\right)\right]$ for the equivalence class of balanced triples of oriented ideals. The set of all equivalence classes of balanced triples of oriented ideals together with ideal multiplication forms a group; we will denote this group by $\Bal$, i.e.,
$$\Bal=\left\{\big[ \left(\mathfrak{I}_1, \mathfrak{I}_2, \mathfrak{I}_3\right)\big]~\left|~ 
\left(\mathfrak{I}_1, \mathfrak{I}_2, \mathfrak{I}_3\right) \text{ a balanced triple of $\OK{L}$-ideals}  \right.\right\}.$$
On the other hand, by $\left([\mathfrak{J}_1], [\mathfrak{J}_2], [\mathfrak{J}_3]\right)$ we will mean a triple of equivalence classes of oriented ideals, i.e., a triple of elements of $\OrelCl{L/K}$. We restrict to triples such that $[\mathfrak{J}_1]\cdot [\mathfrak{J}_2]\cdot [\mathfrak{J}_3]=[(\OK{L};+1, \dots, +1)]$; in other words, $\mathfrak{J}_1\mathfrak{J}_2\mathfrak{J}_3$ is a principal oriented ideal for any choice of the representatives. Again, the set of all such triples forms a group; we will denote this group by $\TripI$, i.e.,
$$\TripI= \left\{ \big([\mathfrak{J}_1], [\mathfrak{J}_2], [\mathfrak{J}_3]\big) ~\left|~  
\exists \ \omega\in L \text{ s.t. } \mathfrak{J}_1\mathfrak{J}_2\mathfrak{J}_3=\OrelP{\omega} \right.\right\}.$$
We will show that these two groups, $\Bal$ and $\TripI$, are isomorphic.

\begin{proposition}\label{Prop:BalTripIso}
The maps
$$\begin{array}{ccc}
\Bal & \longleftrightarrow & \TripI
\vspace{3mm}\\ 
\big[\left(\mathfrak{I}_1, \mathfrak{I}_2, \mathfrak{I}_3\right)\big] & \stackrel{\varphi_1}{\longmapsto} & \big([\mathfrak{I}_1], [\mathfrak{I}_2], [\mathfrak{I}_3]\big),\vspace{4mm}\\
\big[\left(\frac{1}{\omega}\mathfrak{J}_1, \mathfrak{J}_2, \mathfrak{J}_3\right)\big] & \stackrel{\varphi_2}{\longmapsfrom} & 
\begin{array}{c}
\big([\mathfrak{J}_1], [\mathfrak{J}_2], [\mathfrak{J}_3]\big), \\
\mathfrak{J}_1\mathfrak{J}_2\mathfrak{J}_3=\OrelP{\omega}
\end{array}
\end{array}$$

are mutually inverse group homomorphisms.
\end{proposition}

\begin{proof}
Checking that $\varphi_1\circ\varphi_2=\id$ and $\varphi_2\circ\varphi_1=\id$ is easy, so we only need to show that the maps $\varphi_1$ and $\varphi_2$ are well-defined.

We start with the map $\varphi_1$: If a triple $\left(\mathfrak{I}_1, \mathfrak{I}_2, \mathfrak{I}_3\right)$ is balanced, then by the definition $\mathfrak{I}_1 \mathfrak{I}_2 \mathfrak{I}_3$ is the principal oriented ideal $(\OK{L};+1, \dots, +1)$. If we have two representatives of the same class in $\Bal$, $\left(\mathfrak{I}'_1, \mathfrak{I}'_2, \mathfrak{I}'_3\right)\sim\left(\mathfrak{I}_1, \mathfrak{I}_2, \mathfrak{I}_3\right)$, then $\mathfrak{I}'_i\sim\mathfrak{I}_i$, and hence $\big([\mathfrak{I}'_1], [\mathfrak{I}'_2], [\mathfrak{I}'_3]\big)=\big([\mathfrak{I}_1], [\mathfrak{I}_2], [\mathfrak{I}_3]\big)$.

To show that $\varphi_2$ is well-defined, assume that the oriented ideals $\mathfrak{J}_1$, $\mathfrak{J}_2$, $\mathfrak{J}_3$ satisfy $\mathfrak{J}_1\mathfrak{J}_2\mathfrak{J}_3=\OrelP{\omega}$; then $\frac{1}{\omega}\mathfrak{J}_1\mathfrak{J}_2\mathfrak{J}_3=(\OK{L};+1, \dots, +1)$. The definition of $\varphi_2$ does not depend on  the choice of the generator $\omega$ of the principal ideal, because any other generator has to be of the form $\mu\omega$ with $\mu\in\USet{L}$ and $\norm{L/K}{\mu}\in\UPlusSet{K}$, and hence $\left[\left(\frac{1}{\omega}\mathfrak{J}_1, \mathfrak{J}_2, \mathfrak{J}_3\right)\right]=\left[ \left(\frac{1}{\mu\omega}\mathfrak{J}_1, \mathfrak{J}_2, \mathfrak{J}_3\right)\right]$. Moreover, if we take other representatives,  $\big([\lambda_1\mathfrak{J}_1], [\lambda_2\mathfrak{J}_2], [\lambda_3\mathfrak{J}_3]\big)=\big([\mathfrak{J}_1], [\mathfrak{J}_2], [\mathfrak{J}_3]\big)$, then 
$$\varphi_2 \Big(\big([\lambda_1\mathfrak{J}_1], [\lambda_2\mathfrak{J}_2], [\lambda_3\mathfrak{J}_3]\big)\Big) 
= \left[\left(\frac{1}{\lambda_1\lambda_2\lambda_3\omega}\lambda_1\mathfrak{J}_1, \lambda_2\mathfrak{J}_2, \lambda_3\mathfrak{J}_3 \right)\right]
= \left[\left(\frac{1}{\omega}\mathfrak{J}_1, \mathfrak{J}_2, \mathfrak{J}_3 \right)\right].$$
Therefore, the map $\varphi_2$ does not depend on the choice of the representative, and so it is well-defined.
\end{proof}

\begin{proposition}\label{Prop:TripIiso}
The group $\TripI$ is isomorphic to $\OrelCl{L/K}\times\OrelCl{L/K}$.
\end{proposition}

\begin{proof}
The projection 
$$\begin{array}{ccc}
\TripI & \longrightarrow & \OrelCl{L/K}\times\OrelCl{L/K} \vspace{2mm}\\
\big([\mathfrak{J}_1], [\mathfrak{J}_2], [\mathfrak{J}_3]\big) & \longmapsto & \big([\mathfrak{J}_1], [\mathfrak{J}_2]\big)
\end{array}$$
is a group isomorphism, because for given $[\mathfrak{J}_1]$ and $[\mathfrak{J}_2]$, the equivalence class $[\mathfrak{J}_3]$ is given uniquely as $\left[(\mathfrak{J}_1\mathfrak{J}_2)^{-1} \right]$.
\end{proof}

\subsection{From ideals to cubes}\label{Subsec:ItoC}

Finally, we have prepared all the ingredients, and so we can start with the construction of the correspondence between cubes and ideals. As the first step, we will build (an equivalence class of) a cube from a balanced triple of ideals. In this section, we will closely follow the ideas of \cite[Sec.~3.3]{BhargavaLawsI}.

For $\alpha\in L$ define 
$$\tau(\alpha)=\frac{\alpha-\cjg{\alpha}}{\Omega-\cjg{\Omega}}.$$
If $\alpha=a+b\Omega$ for some $a,b\in K$, then the definition of $\tau$ says that $\tau(\alpha)=b$. 
It follows that $\tau$ is additive: For any $\alpha,\beta\in L$, it holds that
$$\tau(\alpha+\beta)=\tau(\alpha)+\tau(\beta).$$
Furthermore, note that if $\OrelIbasis{\alpha,\beta}{M}$ is an oriented ideal, then we can express $\det M$ as $\tau\left(\cjg{\alpha}\beta\right)$. For an oriented ideal, we will often write $[\alpha, \beta]$ instead of $\OrelIbasis{\alpha,\beta}{M}$, the orientation of the ideal implicitly given.

We would like to construct a cube from a given balanced triple of oriented ideals. Consider the following map:
\begin{equation}\label{Eq:PhiDef}\begin{array}{cccc}
\Phi': & \Bal & \longrightarrow & \CSet \vspace{2mm}\\
& \big( [\alpha_1, \alpha_2], [\beta_1, \beta_2], [\gamma_1, \gamma_2] \big) &
\longmapsto & \big(\tau(\alpha_i\beta_j\gamma_k)\big)
\end{array}\end{equation}
(on both sides, the equivalence classes are omitted). The aim of this section is to prove that this map is well-defined. First, we will show how an action by $\widetilde{\Gamma}$ on bases of triples of ideals translates to an action on cubes.

\begin{lemma}\label{Lemma:ChangeBasis}
If $A$ is the image  of $\big( [\alpha_1, \alpha_2], [\beta_1, \beta_2], [\gamma_1, \gamma_2] \big)$ under the map $\Phi'$, then the image of 
$$\big( [p_1\alpha_1+r_1\alpha_2, q_1\alpha_1+s_1\alpha_2], [p_2\beta_1+r_2\beta_2, q_2\beta_1+s_2\beta_2], [p_3\gamma_1+r_3\gamma_3, q_3\gamma_1+s_3\gamma_2] \big)$$
under the map $\Phi'$ is the cube 
$$\left(\begin{psmallmatrix}p_1 &r_1\\ q_1&s_1\end{psmallmatrix} \times \begin{psmallmatrix}p_2 &r_2\\ q_2&s_2\end{psmallmatrix} \times \begin{psmallmatrix}p_3 &r_3\\ q_3&s_3\end{psmallmatrix} \right)(A).$$
\end{lemma}

\begin{proof}
We will consider only the pair of balanced triples $\big( [\alpha_1, \alpha_2], [\beta_1, \beta_2], [\gamma_1, \gamma_2] \big)$ and $\big( [p\alpha_1+r\alpha_2, q\alpha_1+s\alpha_2], [\beta_1, \beta_2], [\gamma_1, \gamma_2] \big)$; the rest follows from \eqref{Eq:ActionDecomp} and the symmetry. Denote $a_{ijk}=\tau(\alpha_i\beta_j\gamma_k)$, and let $(b_{ijk})$ be the cube which arises as the image of the balanced triple $\big( [p\alpha_1+r\alpha_2, q\alpha_1+s\alpha_2], [\beta_1, \beta_2], [\gamma_1, \gamma_2] \big)$ under the map $\Phi'$; then
\begin{align*}
b_{1jk}&=\tau((p\alpha_1+r\alpha_2)\beta_j\gamma_k)=pa_{1jk}+ra_{2jk},\\
b_{2jk}&=\tau((q\alpha_1+s\alpha_2)\beta_j\gamma_k)=qa_{1jk}+sa_{2jk},
\end{align*}
for any $j,k\in\{1,2\}$. Therefore,
\[(b_{ijk})=\left(\prqs\times\id\times\id \right)\big( (a_{ijk}) \big). \qedhere\]
\end{proof}

In the case over $\Z$, Bhargava \cite{BhargavaLawsI} says (in the proof of Theorem 11) that if the balanced triple of ideals is replaced by an equivalent triple, the resulting cube does not change. That is not completely true; as we will prove, the two resulting cubes indeed lie in the same equivalence class. But we cannot expect them to be equal, as shows the following example.

\begin{example}
Assume $K=\Q$, $L=\Q(\ii)$; then $\OK{L}=[1,\ii]$. Both of the triples of oriented ideals $B=\big( [1,\ii], [1,\ii], [1, \ii] \big)$ and $B'=\big( [\ii,-1], [1,\ii], [1, \ii] \big)$ are balanced; moreover, $B\sim B'$, because $[\ii, -1]=\ii\cdot[1, \ii]$, and $\ii$ is a unit from $L$ with (totally) positive norm. We have
$$\Phi'(B)=\cube0110100{-1}, \ \ \ \ \Phi'(B')=\cube100{-1}0{-1}{-1}0;$$
hence, we can see that $\Phi'(B)\neq\Phi'(B')$. But one can check that the cubes are equivalent under the action by $\id\times\begin{psmallmatrix} 0 &1 \\ -1 & 0 \end{psmallmatrix}\times\id$. 
\end{example}

\begin{proposition}
The map $\Phi'$ does not depend on the choice of the representative $\big( [\alpha_1, \alpha_2], [\beta_1, \beta_2], [\gamma_1, \gamma_2] \big)$ of a class in $\Bal$.
\end{proposition}

\begin{proof}
First, we will prove that $\Phi'$ does not depend on the choice of the bases of the oriented ideals. If $[\omega_1, \omega_2]$ and $[p\omega_1+r\omega_2, q\omega_1+s\omega_2]$, $p,q,r,s\in\OK{K}$, are two bases of the same oriented ideal (in particular, with the same orientation), and $M$, $\widetilde{M}$, resp., the corresponding matrices, then $\det\widetilde{M}=(ps-qr)\det M$ and $\barsgn{\det\widetilde{M}}=\barsgn{\det M}$; thus necessarily $ps-qr\in\UPlusSet{K}$. Hence, consider a balanced triple $\big( [\alpha_1, \alpha_2], [\beta_1, \beta_2], [\gamma_1, \gamma_2] \big)$, and let $\big( [p_1\alpha_1+r_1\alpha_2, q_1\alpha_1+s_1\alpha_2], [p_2\beta_1+r_2\beta_2, q_2\beta_1+s_2\beta_2], [p_3\gamma_1+r_3\gamma_2, q_3\gamma_1+s_3\gamma_2] \big)$ be the same balanced triple with other choices of bases. Then $p_is_i-q_ir_i\in\UPlusSet{K}$, $1\leq i\leq3$, and it follows from Lemma~\ref{Lemma:ChangeBasis} that the images of these two balanced triples under the map $\Phi'$ are equivalent cubes.

Now, let $B=\left(\mathfrak{I}_1, \mathfrak{I}_2, \mathfrak{I}_3\right)$ and $B'=\left(\mathfrak{I}'_1, \mathfrak{I}'_2, \mathfrak{I}'_3\right)$ be two equivalent balanced triples. First, assume that $\mathfrak{I}_2=\mathfrak{I}'_2$ and $\mathfrak{I}_3=\mathfrak{I}'_3$; then there exists $\mu\in\USet{L}$ with $\norm{L/K}{\mu}\in\UPlusSet{K}$ and $\mathfrak{I}'_1=\mu\mathfrak{I}_1$. If $\mathfrak{I}_1=\OrelIbasis{\alpha_1,\alpha_2}{M_1}$, then $\mathfrak{I}'_1=\OrelIbasis{\mu\alpha_1,\mu\alpha_2}{M_1}$, and there exist $p,q,r,s\in\OK{K}$ such that $\mu\alpha_1=p\alpha_1+r\alpha_2$ and $\mu\alpha_2=q\alpha_1+s\alpha_2$. Comparing the two expressions of the norm of $\mathfrak{I}'_1$ regarding the bases $[\mu\alpha_1,\mu\alpha_2]$ and $[p\alpha_1+r\alpha_2,q\alpha_1+s\alpha_2]$, we get that
\begin{multline*}
\norm{L/K}{\mu}\det M_1=\norm{L/K}{[\mu\alpha_1,\mu\alpha_2]}\\=\norm{L/K}{[p\alpha_1+r\alpha_2,q\alpha_1+s\alpha_2]}=(ps-qr)\det M_1;
\end{multline*}
hence, $ps-qr=\norm{L/K}{\mu}$, and thus $ps-qr\in\UPlusSet{K}$. By the first part of the proof, the images of $\left(\mathfrak{I}_1, \mathfrak{I}_2, \mathfrak{I}_3\right)$ and $\left(\mu\mathfrak{I}_1, \mathfrak{I}_2, \mathfrak{I}_3\right)$ are equivalent cubes.

Finally, consider the general case: Assume $\mathfrak{I}'_i=\kappa_i\mathfrak{I}_i$ for some $\kappa_i\in L$, $1\leq i \leq 3$. If $B=\big( [\alpha_1, \alpha_2], [\beta_1, \beta_2], [\gamma_1, \gamma_2] \big)$, then $B'=\big( [\kappa_1\alpha_1, \kappa_1\alpha_2], [\kappa_2\beta_1, \kappa_2\beta_2], [\kappa_3\gamma_1, \kappa_3\gamma_2] \big)$. We have that 
\begin{align*}
\Phi'(B)&=\big(\tau(\alpha_i\beta_j\gamma_k)\big),\\
\Phi'(B')&=\big(\tau((\kappa_1\alpha_i)(\kappa_2\beta_j)(\kappa_3\gamma_k))\big)=\big(\tau((\kappa_1\kappa_2\kappa_3)\alpha_i\beta_j\gamma_k)\big).
\end{align*}
Therefore, $\Phi'(B')=\Phi'\big(\left((\kappa_1\kappa_2\kappa_3)\mathfrak{I}_1, \mathfrak{I}_2, \mathfrak{I}_3\right)\big)$, which is equivalent to $\Phi'(B)$ by the previous part of the proof, because $\kappa_1\kappa_2\kappa_3$ is a unit in $L$ of a totally positive norm. 
\end{proof}

\begin{proposition}
Let $B \in \Bal$. Then $\Disc{\Phi'(B)}\in\DSet$, and the cube $\Phi'(B)$ is projective.
\end{proposition}

\begin{proof}
First, assume that $B=\left( [1, \Omega],[1, \Omega],[1, \Omega] \right)$. Then
$$\Phi'(B)=\cubeid{}=A_{\id};$$
therefore, $\Disc{\Phi'(B)}=\left(\Omega-\cjg{\Omega}\right)^2\in\DSet$.


Now, let 
$$B=\big( \OrelIbasis{\alpha_1, \alpha_2}{M_1}, \OrelIbasis{\beta_1, \beta_2}{M_2}, \OrelIbasis{\gamma_1, \gamma_2}{M_3} \big).$$
Then there exists $u\in\UPlusSet{K}$ such that $\det M_1 \det M_2 \det M_3=u$. Moreover, recall that
$$\begin{pmatrix}\alpha_1\\ \alpha_2\end{pmatrix}=M_1\cdot\begin{pmatrix}1\\ \Omega\end{pmatrix}, \ \ \ \ 
\begin{pmatrix}\beta_1\\ \beta_2\end{pmatrix}=M_2\cdot\begin{pmatrix}1\\ \Omega\end{pmatrix}, \ \ \ \ 
\begin{pmatrix}\gamma_1\\ \gamma_2\end{pmatrix}=M_3\cdot\begin{pmatrix}1\\ \Omega\end{pmatrix}.$$
It follows from Lemma~\ref{Lemma:ChangeBasis} that 
$$\Phi'(B)=\left(M_1\times M_2 \times M_3\right)(A_{\id}).$$
Therefore, by Lemma~\ref{Lemma:ChangeDisc}, the discriminant of the cube $\Phi'(B)$ is equal to 
$$(\det M_1)^2(\det M_2)^2(\det M_3)^2\Disc\left(A_{\id}\right)=u^2\left(\Omega-\cjg{\Omega}\right)^2,$$ 
and thus $\Disc{\Phi'(B)}\in\DSet$.

It remains to prove that the cube $\Phi'(B)$ is projective. 
If the cube $\Phi'(B)$ were not projective, then all the coefficients of the three assigned quadratic forms would be divisible by a prime $p\in\OK{K}$ (in fact, by a square of a prime), and $\frac{\Disc{\Phi'(B)}}{u^2p^2}=\frac{D_{\Omega}}{p^2}$ would be a quadratic residue modulo 4 in $\OK{K}$; that would contradict the fact that $D_\Omega$ is fundamental.
\end{proof}

\subsection{From cubes to ideals}\label{Subsec:CtoI}

As the second step in the construction of the correspondence, we need to recover a balanced triple of oriented ideals from a given cube $\left(a_{ijk}\right)$. For this purpose, Bhargava \cite[Proof~of~Thm.~11]{BhargavaLawsI} solves a system of equations
\begin{gather*}
\alpha_i\beta_j\gamma_k=c_{ijk}+a_{ijk}\Omega \\
\alpha_i\beta_j\gamma_k\cdot \alpha_{i'}\beta_{j'}\gamma_{k'}
		=\alpha_{i'}\beta_{j}\gamma_{k}\cdot \alpha_i\beta_{j'}\gamma_{k'}
		=\alpha_{i}\beta_{j'}\gamma_{k}\cdot \alpha_{i'}\beta_j\gamma_{k'}
		=\alpha_{i}\beta_{j}\gamma_{k'}\cdot \alpha_{i'}\beta_{j'}\gamma_k,
\end{gather*}
with $1\leq i, i', j, j', k, k' \leq 2$ and indeterminates $\alpha_i, \beta_j, \gamma_k, c_{ijk}$. That is computationally difficult with $a_{ijk}\in\Z$ already\footnote{\cite[App.~1.1]{Bouyer} mentions a program in Sage running for 24 hours.}, and, to our knowledge, impossible over $\OK{K}$ in general, so we will not follow this path. Instead of that, we will use our results from Theorems~\ref{Theorem:Bijection} and~\ref{Theorem:QFCompFromCube}, i.e., the bijection (group isomorphism) between $\OrelCl{L/K}$ and $\QFSet$, and the fact that the composition of the three quadratic forms arising from one cube lies within the same class as $Q_{\id}$; denote
$$\TripQF=\left\{\left.\big([Q_1], [Q_2], [Q_3]\big) ~\right|~ [Q_i] \in \QFSet, \ 1\leq i\leq3, \ [Q_1Q_2Q_3]=[Q_{\id}] \right\}.$$
It follows from Theorem~\ref{Theorem:QFCompFromCube} that the map $\Theta$ (defined as a restriction to the equivalence classes of the map $\widetilde\Theta$ which assigns the triple of quadratic forms to a cube, see \eqref{Eq:ThetaDef}) actually goes into $\TripQF$. The set $\TripQF$ together with the operation of componentwise composition forms a group. It is clear that if the classes $[Q_1],[Q_2]$ are given, then the class $[Q_3]$ is determined uniquely. Hence, similarly to Proposition~\ref{Prop:TripIiso}, we get:

\begin{proposition}\label{Prop:TripQFiso}
The group $\TripQF$ is isomorphic to the group $\QFSet\times\QFSet$.
\end{proposition}

Proposition~\ref{Prop:TripQFiso} was the last missing piece to close the cycle of maps as illustrated in Figure~\ref{Fig:DetailedScheme}.

\begin{figure}[ht]
\tikzstyle{oedge}=[->,shorten <=2pt, shorten >=3pt,>=angle 90]
\begin{tikzpicture}[scale=1.5]

\node (CD) at (0,0) {$\CSet$};
\node (TQF) at (3,0) {$\TripQF$};
\node (QFQF) at (6,0) {$\QFSet\times\QFSet$};
\node (ClCl) at (6,1) {$\OrelCl{L/K}\times\OrelCl{L/K}$};
\node (TI) at (3,1) {$\TripI$};
\node (Bal) at (0,1) {$\Bal$};

\draw[oedge]  (CD) [->]-- node [below] {\tiny{Th.~\ref{Theorem:QFCompFromCube}}}  (TQF);
\draw[oedge] (TQF) [<->] -- node [above] {$\simeq$} node [below] {\tiny{Prop.~\ref{Prop:TripQFiso}}}  (QFQF);
\draw[oedge] (QFQF)[->] to [out=35,in=315] node [right] {\footnotesize{$\Psi\times\Psi$}}  (ClCl);
\draw[oedge] (ClCl) [->] to [out=215,in=145] node [left] {\footnotesize{$\Phi\times\Phi$}} (QFQF); 
\draw[oedge] (ClCl) [<->] -- node [above] {$\simeq$} node [below] {\tiny{Prop.~\ref{Prop:TripIiso}}} (TI);
\draw[oedge] (TI) [<->] -- node [above] {$\simeq$} node [below] {\tiny{Prop.~\ref{Prop:BalTripIso}}} (Bal);
\draw[oedge] (Bal) [->] to [out=225,in=135] node [left] {\footnotesize{$\Phi'$}} (CD);
\draw[oedge] (CD) [->, dashed] to [out=45,in=305] node [right] {\footnotesize{$\Psi'$}} (Bal);

\node (text1) at (6, 0.55) {\tiny{1:1}};
\node (text2) at (6, 0.40) {\tiny{Th.~\ref{Theorem:Bijection}}};
	
\end{tikzpicture}
\caption{Diagram of the construction of the map $\Psi'$.}
\label{Fig:DetailedScheme}
\end{figure}

Going around the cycle in the diagram, we can formally define a map 
\begin{equation}\label{Eq:PsiDef}\begin{array}{cccc}
\Psi': &\CSet& \longrightarrow &\Bal\vspace{2mm}\\
& [A] & \longmapsto & \Big[\big(\Psi(Q_1),\Psi(Q_2), \left(\Psi(Q_1)\Psi(Q_2)\right)^{-1}\big) \Big],
\end{array}\end{equation}
where $Q_1$, $Q_2$ (and $Q_3$) are the quadratic forms arising from the cube $A$. This map is well-defined, because all the maps on the way are well-defined. 

\subsection{Conclusion}\label{Subsec:Conclusion}

At this point, just one last step is remaining, and that is to prove that both $\Psi'\circ\Phi'$ and $\Phi'\circ\Psi'$ are the identity maps. It is clear from Figure~\ref{Fig:DetailedScheme} that the map $\Psi\times\Psi\times\Psi$ (and also the map $\Phi\times\Phi\times\Phi$) provide a group isomorphism between $\TripI$ and $\TripQF$. In Figure~\ref{Fig:SimplifiedScheme}, we provide a simplification of the diagram displayed in Figure~\ref{Fig:DetailedScheme}.

\begin{figure}[h]
\tikzstyle{oedge}=[->,shorten <=2pt, shorten >=3pt,>=angle 90]
\begin{tikzpicture}[scale=1.5]

\node (CD) at (0,0) {$\CSet$};
\node (TQF) at (3,0) {$\TripQF$};
\node (TI) at (3,1) {$\TripI$};
\node (Bal) at (0,1) {$\Bal$};

\draw[oedge]  (CD) [->]-- node [above] {\footnotesize{$\Theta$}} (TQF);
\draw[oedge] (TI) [->] to [out=225,in=135] node [left] {\footnotesize{$\Phi\times\Phi\times\Phi$}} (TQF);
\draw[oedge] (TQF) [->] to [out=45,in=305] node [right] {\footnotesize{$\Psi\times\Psi\times\Psi$}}  (TI);
\draw[oedge] (TI) [<->] --  (Bal);
\draw[oedge] (Bal) [->] to [out=225,in=135] node [left] {\footnotesize{$\Phi'$}} (CD);
\draw[oedge] (CD) [->] to [out=45,in=305] node [right] {\footnotesize{$\Psi'$}} (Bal);

\end{tikzpicture}
\caption{Simplified diagram of the construction of the map $\Psi'$.}
\label{Fig:SimplifiedScheme}
\end{figure}

\begin{proposition}\label{Prop:PhiPsi}
$\Phi'\circ\Psi'$ is the identity map on $\CSet$.
\end{proposition}

\begin{proof}
First, note that by using the result of Lemma~\ref{Lemma:ProductOfIdeals} and the isomorphism $\varphi_2$ from Proposition~\ref{Prop:BalTripIso}, we could have defined the map $\Psi'$ equivalently as
$$\begin{array}{cccc}
\Psi': &\CSet& \longrightarrow &\Bal\vspace{2mm}\\
& [A] & \longmapsto & \varphi_2\Big(\big([\Psi(Q_1)], [\Psi(Q_2)], [\Psi(Q_3)]\big) \Big);
\end{array}$$
the situation is depicted in Figure~\ref{Fig:PhiPsi}.

\begin{figure}[ht]
\tikzstyle{oedge}=[->,shorten <=2pt, shorten >=3pt,>=angle 90]
\begin{tikzpicture}[scale=1.1]

\node (A) at (-1,-1) {$\textcolor[rgb]{0.4,0.4,0.4}{[A]}$};
\node (QQQ) at (4.1, -1) {$\textcolor[rgb]{0.4,0.4,0.4}{\big([Q_1], [Q_2], [Q_3]\big)}$};
\node (PsiQ) at (4.1,3) {$\textcolor[rgb]{0.4,0.4,0.4}{\big([\Psi(Q_1)],[\Psi(Q_2)], [\Psi(Q_3)]\big)}$};
\node (PsiA) at (-1,3) {$\textcolor[rgb]{0.4,0.4,0.4}{\big[\left(\frac{1}{\omega}\Psi(Q_1),\Psi(Q_2), \Psi(Q_3)\right)\big]}$};

\draw[oedge] (A)[|->, opacity=0.6] -- (QQQ);
\draw[oedge] (QQQ)[|->, opacity=0.6] -- (PsiQ);
\draw[oedge] (PsiQ)[|->, opacity=0.6] -- (PsiA);
\draw[oedge] (A)[|->, opacity=0.6, dashed] -- (PsiA);

\node (CD) at (0,0) {$\CSet$};
\node (TQF) at (3,0) {$\TripQF$};
\node (TI) at (3,2) {$\TripI$};
\node (Bal) at (0,2) {$\Bal$};

\draw[oedge]  (CD) [->]-- node [above] {\footnotesize{$\Theta$}} (TQF);
\draw[oedge] (TQF) [->] -- node [left] {\footnotesize{$\Psi\times\Psi\times\Psi$}}  (TI);
\draw[oedge] (TI) [->] -- node [above] {\footnotesize{$\varphi_2$}}  (Bal);
\draw[oedge] (CD) [->, dashed] -- node [right] {\footnotesize{$\Psi'$}} (Bal);

\end{tikzpicture}
\caption{Situation in the proof of Proposition~\ref{Prop:PhiPsi}.}
\label{Fig:PhiPsi}
\end{figure}

Let $A$ be a representative of a class in $\CSet$. By Lemma~\ref{Lemma:CubeReduction}, we can assume without loss of generality that $A$ is reduced, i.e., 
$$A=\cubered[1.1].$$ Then
\begin{equation*}\begin{aligned}
\Psi(Q_1)&=\left(\left[-d, \frac{-h+\sqrt{D}}{2} \right]; \barsgn{-d} \right),\\
\Psi(Q_2)&=\left(\left[-g, \frac{-h+\sqrt{D}}{2} \right]; \barsgn{-g} \right),\\
\Psi(Q_3)&=\left(\left[-f, \frac{-h+\sqrt{D}}{2} \right]; \barsgn{-f} \right),\\
\end{aligned}\end{equation*}
and we have proved in Lemma~\ref{Lemma:ProductOfIdeals} that
$$\Psi(Q_1)\Psi(Q_2)\Psi(Q_3)=\left(\left(\frac{-h+\sqrt{D}}{2}\right); \barsgn{-dfg}\right),$$
where $D=\Disc(A)=h^2+4dfg$. Hence, if we apply the map $\varphi_2$ to the triple $\big([\Psi(Q_1)], [\Psi(Q_2)], [\Psi(Q_3)]\big)$, we obtain $\big[\left(\mathfrak{J}_1, \mathfrak{J}_2, \mathfrak{J}_3 \right)\big]$, where
$$\begin{aligned}
\mathfrak{J}_1&=\left(\left[\frac{-h-\sqrt{D}}{2fg}, 1 \right]; \barsgn{\frac{1}{fg}} \right),\\
\mathfrak{J}_2&=\left(\left[-g, \frac{-h+\sqrt{D}}{2} \right]; \barsgn{-g} \right),\\
\mathfrak{J}_3&=\left(\left[-f, \frac{-h+\sqrt{D}}{2} \right]; \barsgn{-f} \right).\\
\end{aligned}$$
Denote $B=\Phi'\Big(\big[\left(\mathfrak{J}_1, \mathfrak{J}_2, \mathfrak{J}_3 \right)\big]\Big)$. Using the facts that $D=u^2D_{\Omega}$ for a totally positive unit $u\in\UPlusSet{K}$, and that $\Omega=\frac{-w+\sqrt{D_{\Omega}}}{2}$, we compute that 
$$B=\cube{-u}00{-ud}0{-uh}{-ug}{-uh}{}$$ 
(see Table~\ref{Tab:CubeComp} for detailed computations). Noting that $B=-uA$, we have that $[A]=[B]$, and hence $\Phi'\circ\Psi'=\id_{\CSet}$.

\begin{table}
\setlength\extrarowheight{8pt}
\begin{tabular}{C|L|C} 
ijk & \text{Argument of the map } \tau & b_{ijk}\\[1pt]
\hline
111 & \frac{-h-\sqrt{D}}{2fg}(-g)(-f)=-\frac{h+\sqrt{D}}{2}=-\frac{uw+h}{2}-u\Omega &-u\\
112 & \frac{-h-\sqrt{D}}{2fg}(-g)\frac{-h+\sqrt{D}}{2} = \frac{D-h^2}{4f}=-dg & 0\\
121 & \frac{-h-\sqrt{D}}{2fg}\frac{-h+\sqrt{D}}{2}(-f) = \frac{D-h^2}{4g}=-df & 0\\
211 & fg & 0\\
122 & \frac{-h-\sqrt{D}}{2fg}\frac{-h+\sqrt{D}}{2}\frac{-h+\sqrt{D}}{2}=\frac{h^2-D}{4fg}\frac{-h+\sqrt{D}}{2}=-d\frac{-h+\sqrt{D}}{2}=d\frac{h-uw}{2} -ud\Omega& -ud\\
212 & (-g)\frac{-h+\sqrt{D}}{2}=g\frac{h-uw}{2}-ug\Omega & -ug \\
221 & \frac{-h+\sqrt{D}}{2}(-f)=f\frac{h-uw}{2}-uf\Omega & -uf \\
222 & \frac{-h+\sqrt{D}}{2}\frac{-h+\sqrt{D}}{2}=\frac{D+h^2}{4}-h\frac{\sqrt{D}}{2}=\frac{h^2+2dfg-uhw}{2}-hu\Omega & -uh\\
\end{tabular}
\vspace{3mm}
\caption{Computation of the cube $(b_{ijk})=\Phi'\Big(\big[\left(\mathfrak{J}_1, \mathfrak{J}_2, \mathfrak{J}_3 \right)\big]\Big)$.}
\label{Tab:CubeComp}
\end{table}
\end{proof}

\begin{proposition}\label{Prop:PsiPhi}
$\Psi'\circ\Phi'$ is the identity map on $\Bal$.
\end{proposition}

\begin{proof}
We broadly follow the ideas of \cite[Sec.~5.4]{notes}. 
Let $\rho_1$ be the isomorphism between $\Bal$ and $\TripQF$ given by the map $(\Phi\times\Phi\times\Phi)\circ\varphi_1$, where $\varphi_1$ is the map defined in Proposition~\ref{Prop:BalTripIso}, and denote by $\rho_2$ the map $\Theta\circ\Phi'$ (see Figure~\ref{Fig:PsiPhi}); recall that the map $\Theta$ has been defined in \eqref{Eq:ThetaDef}. We will show that $\rho_1=\rho_2$; then $\Psi'\circ\Phi'=\id_{\mathcal{B}al\left(\OrelCl{L/K}\right)}$ follows, since $\Psi'\circ\Phi'=\rho_1^{-1}\circ\rho_2$.

\begin{figure}[ht]
\tikzstyle{oedge}=[->,shorten <=2pt, shorten >=3pt,>=angle 90]
\begin{tikzpicture}[scale=1.5]

\node (CD) at (0,0) {$\CSet$};
\node (TQF) at (3,0) {$\TripQF$};
\node (TI) at (3,1) {$\TripI$};
\node (Bal) at (0,1) {$\Bal$};
\node at (2.25,0.75) {$\rho_1$};
\node at (0.75,0.25) {$\rho_2$};

\draw[oedge] (CD) [->]-- node [below] {\footnotesize{$\Theta$}} (TQF);
\draw[oedge] (TI) [->] -- node [right] {\footnotesize{$\Phi\times\Phi\times\Phi$}}  (TQF);
\draw[oedge] (TI) [<-] -- node [above] {\footnotesize{$\varphi_1$}}  (Bal);
\draw[oedge] (Bal) [->] -- node [left] {\footnotesize{$\Phi'$}} (CD);
	
\draw[oedge] (Bal) [->, dashed] .. controls (2.2,0.7) ..(TQF);
\draw[oedge] (Bal) [->, dashed] .. controls (0.8,0.3)..(TQF);

\end{tikzpicture}
\caption{Maps $\rho_1$ and $\rho_2$.}
\label{Fig:PsiPhi}
\end{figure}

Consider a balanced triple $B\in\Bal$, and let 
$$B=\Big(\OrelIbasis{\alpha_1,\alpha_2}{M_1}, \OrelIbasis{\beta_1, \beta_2}{M_2}, \OrelIbasis{\gamma_1,\gamma_2}{M_3}\Big).$$
Then $\rho_1(B)$ is equal to
\begin{equation}\label{Eq:rho1}
\left( \left[\frac{\norm{L/K}{\alpha_1x-\alpha_2y}}{\det M_1}\right], \left[\frac{\norm{L/K}{\beta_1x-\beta_2y}}{\det M_2}\right], \left[\frac{\norm{L/K}{\gamma_1x-\gamma_2y}}{\det M_3}\right] \right).
\end{equation}

Let us compute $\rho_2(B)$. Recall that $\Phi'(B)=\big(\tau(\alpha_i\beta_j\gamma_k)\big)$ and $\tau(\alpha)=\frac{\alpha-\cjg{\alpha}}{\Omega-\cjg{\Omega}}$; for $\zeta\in L$, set
$$G(\zeta)=
\begin{pmatrix}  
\tau(\alpha_1\beta_1\zeta) & \tau(\alpha_2\beta_1\zeta)\\
\tau(\alpha_1\beta_2\zeta) & \tau(\alpha_2\beta_2\zeta)\\
\end{pmatrix}.$$
Then 
$$\det G(\zeta) = -\norm{L/K}{\zeta}\cdot\frac{\cjg{\alpha_1}\alpha_2-\alpha_1\cjg{\alpha_2}}{\Omega-\cjg{\Omega}}\cdot\frac{\cjg{\beta_1}\beta_2-\beta_1\cjg{\beta_2}}{\Omega-\cjg{\Omega}} = -\norm{L/K}{\zeta}\det M_1 \det M_2.$$
Note that $G(\gamma_1)$ is the upper face of the cube $\big(\tau(\alpha_i\beta_j\gamma_k)\big)$, and $G(\gamma_2)$ is the lower face; hence, if $Q_3$ denotes the third quadratic form assigned to this cube, it holds that
$$Q_3(x, y)=-\det\left(G(\gamma_1)x-G(\gamma_2)y \right).$$
Therefore,
\begin{equation}\label{Eq:Q3}
Q_3(x, y)=-\det\left(G(\gamma_1x-\gamma_2y) \right)=\norm{L/K}{\gamma_1x-\gamma_2y}\det M_1 \det M_2.
\end{equation}
Since $B$ is a balanced triple of ideals, there exists a totally positive unit $u\in\UPlusSet{K}$ such that $\det M_1 \det M_2\det M_3=u$. Thus, we can rewrite the expression \eqref{Eq:Q3} as
$$Q_3(x, y)=u\cdot\frac{\norm{L/K}{\gamma_1x-\gamma_2y}}{\det M_3}.$$

Similarly, we can prove that 
\begin{align*}
Q_1(x, y)&=u\cdot\frac{\norm{L/K}{\alpha_1x-\alpha_2y}}{\det M_1},\\
Q_2(x, y)&=u\cdot\frac{\norm{L/K}{\beta_1x-\beta_2y}}{\det M_2}.
\end{align*}
Thus, 
$$\rho_2(B)=\big([Q_1], [Q_2], [Q_3]\big);$$
comparing with \eqref{Eq:rho1}, we see that $\rho_1(B)=\rho_2(B)$.
\end{proof}

We can summarize our results; we need the term \emph{fundamental element}, which we have introduced in Definition~\ref{Def:AlmostSquare-free}.

\begin{theorem}\label{Theorem:BijectionCubes}
Let $K$ be a number field of narrow class number one with at least one real embedding. Let $D$ be a~fundamental element of $\OK{K}$. Set $L=K\big(\sqrt{D}\big)$, and $\DSet=\left\{u^2D ~|~ u\in\UPlusSet{K}\right\}$. Then we have a bijection between $\CSet$ and $\Bal$ given by the map $\Phi'$ defined in \eqref{Eq:PhiDef} (equivalently by the map $\Psi'$ defined in \eqref{Eq:PsiDef}).
\end{theorem}

\begin{proof}
Let $\Omega$ be such that $\OK{L}=[1, \Omega]$. Then there exists a unit $u\in\USet{K}$ (not necessarily totally positive) such that $D=u^2D_{\Omega}$; it follows that we have $\OK{L}=[1, u\Omega]$ and $u^2D_{\Omega}=(u\Omega-\cjg{u\Omega})^2$. If we begin with the canonical basis $[1, u\Omega]$ of $\OK{L}$ instead of $[1, \Omega]$, we get the same results for $D$ in the place of $D_{\Omega}$. Therefore, without loss of generality, we may assume that $D=D_{\Omega}$. 
In Propositions~\ref{Prop:PhiPsi} and~\ref{Prop:PsiPhi}, we have proved that $\Phi'$ and $\Psi'$  are mutually inverse bijections.
\end{proof}

\begin{corollary}
$\CSet$ carries a group structure arising from multiplication of ba\-lan\-ced triples of oriented ideals in $K(\sqrt{D})$. The identity element of this group is represented by the cube $A_{\id}$,
$$A_{\id}=\cubeid.$$
The inverse element to $\big[(a_{ijk})\big]$ is $\big[(-1)^{i+j+k}(a_{ijk})\big]$, i.e., 
$$\left[\cubearb[1.1]{}\right]^{-1}=\left[\cube{-a}{b}{c}{-d}{e}{-f}{-g}{h}{} \right].$$
\end{corollary}

\begin{proof}
The group structure of $\CSet$ follows from Theorem~\ref{Theorem:BijectionCubes}. 

Obviously, the triple $\big([1, \Omega], [1, \Omega], [1, \Omega]\big)$ is a representative of the identity element in the group $\Bal$;
its image under the map $\Phi'$ is the class represented by the cube $A_{\id}$.

Consider a cube $(a_{ijk})$; by Theorem~\ref{Theorem:BijectionCubes}, there exists a balanced triple of ideals
$$B=\big( \OrelIbasis{\alpha_1, \alpha_2}{M_1}, \OrelIbasis{\beta_1, \beta_2}{M_2}, \OrelIbasis{\gamma_1, \gamma_2}{M_3} \big)$$
such that $\Phi'(B)=(a_{ijk})$, i.e., $a_{ijk}=\tau(\alpha_i\beta_j\gamma_k)$, $1\leq i, j, k \leq 2$. It follows from Lemma~\ref{Lemma:InverseIdeals} that the balanced triple inverse to $B$ is $B^{-1}=\big(\mathfrak{I}_1, \mathfrak{I}_2, \mathfrak{I}_3\big)$, where
\begin{align*}
\mathfrak{I}_1&=\OrelIbasis{\frac{\cjg{\alpha_1}}{\det M_1}, -\frac{\cjg{\alpha_2}}{\det M_1}}{M_1},\\
\mathfrak{I}_2&= \OrelIbasis{\frac{\cjg{\beta_1}}{\det M_2}, -\frac{\cjg{\beta_2}}{\det M_2}}{M_2},\\
\mathfrak{I}_3&= \OrelIbasis{\frac{\cjg{\gamma_1}}{\det M_3}, -\frac{\cjg{\gamma_2}}{\det M_3}}{M_3}. 
\end{align*}
Denote $u=\det M_1\det M_2\det M_3$. Then 
$$\Phi'\left(B^{-1}\right)=\Big(\tau\left((-1)^{i+j+k+1}u^{-1}\cjg{\alpha_i\beta_j\gamma_k}\right)\Big)=\Big((-1)^{i+j+k+1}u^{-1}\tau\left(\cjg{\alpha_i\beta_j\gamma_k} \right)\Big).$$
Multiplying the cube by $u$, we get an equivalent cube $\Big((-1)^{i+j+k+1}\tau\left(\cjg{\alpha_i\beta_j\gamma_k} \right)\Big)$; hence, we have to compute $\tau\left(\cjg{\alpha_i\beta_j\gamma_k}\right)$. If we write 
$$\alpha_i\beta_j\gamma_k=c_{ijk}+a_{ijk}\Omega,$$
for some $c_{ijk}\in\OK{K}$, $1\leq i, j, k\leq 2$, then
$$\cjg{\alpha_i\beta_j\gamma_k}=c_{ijk}+a_{ijk}\cjg{\Omega}.$$
Since $\Omega=\frac{-w+\sqrt{D_{\Omega}}}{2}$, we have that $\cjg{\Omega}=-\Omega-w$; therefore,
$$\cjg{\alpha_i\beta_j\gamma_k}=(c_{ijk}-a_{ijk}w)-a_{ijk}{\Omega},$$
and
$$\tau\left(\cjg{\alpha_i\beta_j\gamma_k}\right)=(-a_{ijk}).$$
It follows that the cube $\Phi'\left(B^{-1}\right)$ is equivalent to the cube $\Big((-1)^{i+j+k}a_{ijk}\Big)$.
\end{proof}

Together with the isomorphisms we have known before, we have proved the following corollary.

\begin{corollary}
The groups $\CSet$, $\TripQF$ and $\OrelCl{L/K}\times\OrelCl{L/K}$ are isomorphic.
\end{corollary}

We have actually proved even more; see Figure~\ref{Fig:Isos}.

\begin{figure}[ht]
\begin{tikzpicture}[scale=1.5]

\node (CD) at (0,0) {$\CSet$};
\node (TQF) at (3,0) {$\TripQF$};
\node (QFQF) at (6,0) {$\QFSet\times\QFSet$};
\node (ClCl) at (6,1) {$\OrelCl{L/K}\times\OrelCl{L/K}$};
\node (TI) at (3,1) {$\TripI$};
\node (Bal) at (0,1) {$\Bal$};

\draw  (CD) [->]-- node [above] {$\simeq$} node [below] {\tiny{Th.~\ref{Theorem:QFCompFromCube}}} (TQF);
\draw (TQF) [<->] -- node [above] {$\simeq$} node [below] {\tiny{Prop.~\ref{Prop:TripQFiso}}}  (QFQF);
\draw (QFQF)[<->] -- node [left] {$\simeq$}  (ClCl);
\draw (ClCl) [<->] -- node [above] {$\simeq$} node [below] {\tiny{Prop.~\ref{Prop:TripIiso}}} (TI);
\draw (TI) [<->] -- node [above] {$\simeq$} node [below] {\tiny{Prop.~\ref{Prop:BalTripIso}}} (Bal);
\draw (TI) [<->] -- node [left] {$\simeq$} (TQF);
\draw (Bal) [<->] -- node [left] {$\simeq$} node [right]{\tiny{Theorem~\ref{Theorem:BijectionCubes}}} (CD);
	
\end{tikzpicture}
\caption{Summary of the isomorphisms.}
\label{Fig:Isos}
\end{figure}

\appendix \section{} \label{Sec:App}
In Section~5 of \cite{KZforms}, the author explains what causes the problem if $K$ is a totally imaginary number field and provides a solution. Unfortunately, the solution is incorrect: There is a mistake in the proof of Proposition~5.5 of that paper, and here we will show that this mistake cannot be corrected.

Let us review the situation. Assume that $K$ is a totally imaginary number field (i.e., $K$ has no real embeddings) with $h^+(K)=1$, and let $L/K$ be a quadratic extension. The set $\QFSet$ is defined analogously as in Subsetion~\ref{Subsec:QFs} of this article. Note that for the case of totally imaginary $K$, we have $\USet{K}=\UPlusSet{K}$. In particular, $-1$ is a totally positive unit, and so $Q\sim -Q$.

Moreover, note that for any such $K$ and $L$, we have $\OrelCl{L/K}\cong\Cl{L}$. In \cite[Subsec.~5.3]{KZforms}, the subgroup $\Gr{L/K}$ of $\Cl{L}$ is defined by
\[\Gr{L/K}=\left\{I\left(\overline{I}\right)^{-1}\PSet{L}~\big|~I\in\ISet{L}\right\},\]
and then the factor group
\[\ClImag{L/K}=\quotient{\Cl{L}}{\Gr{L/K}}\]
(called \emph{imaginary class group}) is considered.

First problem unobserved in \cite{KZforms} is that the group $\ClImag{L/K}$ is often trivial.

\begin{lemma}\label{Lemma:OddImagClassGroup}
Let $K$ and $L$ be as above, and suppose that $h(L)$ is odd. Then $\big|\ClImag{L/K}\big|=1$.
\end{lemma}

\begin{proof}
Consider the map
\[ \begin{array}{cccc}
f: & \Cl{L} & \longrightarrow & \Cl{L}\\
& J\PSet{L} & \longmapsto & J^2\PSet{L}.
\end{array}\]
Since the order of the group $\Cl{L}$ is odd by the assumption, it follows that $f(J\PSet{L})\in\PSet{L}$ if and only if $J\in\PSet{L}$. Thus, $f$ is injective, and hence a bijection. Therefore, for any $I\in\ISet{L}$, we can find $I_0\in\ISet{L}$ such that $I\PSet{L}=I_0^2\PSet{L}$. Moreover, recall that by Lemma~\ref{Lemma:InverseIdeals}, we have $(\overline{I_0})^{-1}\PSet{L}=I_0\PSet{L}$. It follows that $I\PSet{L}=I_0(\overline{I_0})^{-1}\PSet{L}$, and hence $I\PSet{L}\in\Gr{L/K}$. The claim follows.
\end{proof}

To understand the problem, we provide two quadratic field extensions of $\Q(\ii)$, for which we compute the set $\QFSet$. But first, we need to prepare some lemmas.

\begin{lemma}\label{Lemma:PhiSurjective}
The map
\[
\begin{array}{cccc}
\Phi:&\Cl{L}
&\longrightarrow
& \QFSet \vspace{3mm}\\
&\left[\alpha, \beta\right] & \longmapsto &
\frac{\norm{L/K}{\alpha x+\beta y}}{\frac{\cjg{\alpha}\beta-\alpha\cjg{\beta}}{\Omega-\overline{\Omega}}}=\frac{\alpha\cjg{\alpha}x^2-(\cjg{\alpha}\beta+\alpha\cjg{\beta})xy+\beta\cjg{\beta}y^2}{\frac{\cjg{\alpha}\beta-\alpha\cjg{\beta}}{\Omega-\overline{\Omega}}}   
\end{array}
\] 
is well-defined and surjective.
\end{lemma}

\begin{proof}
The verification that $\Phi$ is well-defined has actually been done in the proof of \cite[Prop.~5.5]{KZforms}. For the surjectivity, it is easy to check that a quadratic form $ax^2+bxy+cy^2$ (with $a,b,c\in\OK{K}$ and $b^2-4ac\in\DSet$) is the image of the ideal $\left[a,\frac{-b+\sqrt{b^2-4ac}}{2}\right]$.
\end{proof}

Before we state the next lemma, note that equivalent quadratic forms over $\OK{K}$ represent the same elements of $\OK{K}$ up to a multiple by some $u\in\USet{K}$.

\begin{lemma} \label{Lemma:RepreJednotky}
Let $K$, $L$, $\Phi$ be as above and let $I\in\ISet{L}$ be integral. Then the quadratic form $Q=\Phi(I)$ represents some $u\in\USet{K}$ if and only if $I\in\PSet{L}$.
\end{lemma}

\begin{proof}
If $I\in\PSet{L}$, then $I\sim[1,\Omega]$, and hence 
\[Q=\Phi(I)\sim\Phi([1,\Omega])=x^2-(\Omega+\cjg{\Omega})xy+\Omega\cjg{\Omega}y^2=Q_{\id}(x,y)\] 
(see Theorem~\ref{Theorem:Bijection}). Since $Q_{\id}$ represents $1$, it follows that $Q$ represents some $u\in\USet{K}$.

To prove the other direction, assume that $I=[\alpha,\beta]$ for some $\alpha, \beta\in\OK{L}$, and that $Q=\Phi(I)$ represents $u\in\USet{K}$. Moreover, denote $m=\frac{\cjg{\alpha}\beta-\alpha\cjg{\beta}}{\Omega-\overline{\Omega}}$; then 
\[Q(x,y)=\frac{\norm{L/K}{\alpha x +\beta y}}{m}.\]
By the assumption, there exist $a,b\in\OK{K}$ such that $Q(a,b)=u$. It follows that $\norm{L/K}{\alpha a +\beta b}=um$. Denote $\gamma=\alpha a +\beta b$. Then $(\gamma)\subseteq I$, and so there exists another integral $\OK{L}$-ideal $J$ such that $IJ=(\gamma)$. Recall that $\norm{L/K}{I}=(m)$ and $\norm{L/K}{(\gamma)}=(\norm{L/K}{\gamma})$ (see Subsection~\ref{Subsec:IdealsAndOrientation}). It follows that
\[(m)\cdot\norm{L/K}{J}=\norm{L/K}{I}\cdot\norm{L/K}{J}=\norm{L/K}{IJ}=\norm{L/K}{(\gamma)}=(um)=(m).\]
Therefore, $\norm{L/K}{J}=\OK{K}$. Furthermore, recall that $\norm{L/K}{J}=(\norm{L/K}{\omega}~|~\omega\in J)$ by the definition. Thus, there exists $\mu\in J$ such that $\norm{L/K}{\mu}=1$. Since the ideal $J$ is integral, we must have $\mu\in\USet{L}$. It follows that $J=\OK{L}$, and so the equality $IJ=(\gamma)$ translates as $I=(\gamma)$.
\end{proof}

\begin{example} \label{Ex:3elClL}
Consider $K=\Q(\ii)$ and $L=\Q(\ii,\sqrt{23})$. Then we have $\OK{L}=[1,\Omega]$ with $\Omega=\frac{\ii+\sqrt{23}}{2}$, and the group $\Cl{L}\cong\Z/3\Z$ is generated by (the class of) the ideal $I=[1-\ii,\Omega]$. Note that since we are in a three-element group, we have $[I]^2=[I]^{-1}$. As $[I]^{-1}=[\cjg{I}]$, we get $[I]^2=[\cjg{I}]$. This means that we have 
\[\Cl{L}=\left\{[\OK{L}], [I], [\cjg{I}]\right\}.\] 
Applying the map $\Phi$ on the elements of $\Cl{L}$, we obtain the following quadratic forms:
\[\begin{array}{ccl}
\OK{L} & \longmapsto & x^2-\ii xy-6y^2 =: Q_0(x,y),\\
I & \longmapsto & (1-\ii)x^2-\ii xy-3(1+\ii)y^2=:Q_1(x,y),\\
I^2\sim\overline{I} & \longmapsto & -(1-\ii)x^2+\ii xy+3(1+\ii)y^2=:Q_2(x,y).
\end{array}\]
We want to establish the size of $\QFSet$. Since the map $\Phi$ is surjective by Lemma~\ref{Lemma:PhiSurjective}, we know that $\QFSet$ contains only classes represented by the forms $Q_0$, $Q_1$, $Q_2$. It is clear that $Q_2=-Q_1$, and hence $Q_1\sim Q_2$. Thus, the question is whether $Q_1$ is equivalent to $Q_0$. Since $Q_1=\Phi(I)$ and $[I]\neq[\OK{L}]$ (otherwise, $[I]$ could not generate $\Cl{L}$), Lemma~\ref{Lemma:RepreJednotky} gives a negative answer to this question. It follows that $\QFSet=\{[Q_0],[Q_1]\}$ with $Q_0\not\sim Q_1$, and hence $\abs{\QFSet}=2$.

On the other hand, we have $\big|\ClImag{L/K}\big|=1$ by Lemma~\ref{Lemma:OddImagClassGroup}, so there does not exist any well-defined surjective map from $\ClImag{L/K}$ to $\QFSet$. Neither can we find a bijection between $\QFSet$ and $\Cl{L}$ or any of its factor group or subgroup, since all of them must be of order $1$ or $3$.
\end{example}

\begin{example} \label{Ex:4elClL}
Let $K=\Q(\ii)$ and $L=\Q(\ii,\sqrt{14})$. Then $\OK{L}=[1,\Omega]$ for $\Omega=\frac{\sqrt{14}+\ii\sqrt{14}}{2}$. We have $\Cl{L}\cong\Z/4\Z$, and this group is generated by (the class of) the ideal $I$, where
\[ I=[2+\ii,\Omega+1], \quad I^2=[3+4\ii,\Omega+1], \quad I^3\sim\cjg{I}=[2+\ii,\cjg{\Omega}+1].\]
If we apply the map $\Phi$ on the representatives of the elements of $\Cl{L}$, we get the following quadratic forms:
\[\begin{array}{ccl}
\OK{L} & \longmapsto & x^2-7\ii y^2 =: Q_0(x,y),\\
I & \longmapsto & (2+\ii)x^2-2 xy-(1+3\ii)y^2=:Q_1(x,y),\\
I^2 & \longmapsto & (3+4\ii)x^2-2xy-(1+\ii)y^2=:Q_2(x,y),\\
I^3 & \longmapsto & -(2+\ii)x^2+2 xy+(1+3\ii)y^2=:Q_3(x,y).
\end{array}\]
Again, we want to find out whether these four forms give pairwise different equivalence classes. First of all, we see that $Q_3=-Q_1$, and hence $Q_3\sim Q_1$. Moreover, a simple computation in Magma tells us that $Q_0$ does not lie in the same genus as $Q_1$, and $Q_1$ is not in the same genus as $Q_2$. It follows that $Q_0\not\sim Q_1$ and $Q_1\not\sim Q_2$. Thus, it only remains to decide whether $Q_0$ is equivalent to $Q_2$. Here we apply Lemma~\ref{Lemma:RepreJednotky}: Since $I^2\notin\PSet{L}$, the quadratic form $Q_2$ does not represent any unit, and hence it cannot be equivalent to $Q_0$. It follows that $\QFSet=\{[Q_0], [Q_1], [Q_2]\}$ and $\abs{\QFSet}=3$. Since any factor group or subgroup of $\Cl{L}$ is of order $1$, $2$ or $4$, we cannot map it bijectively on $\QFSet$.
\end{example}

The main problem of \cite{KZforms} lies in the proof of Proposition~5.5.; in particular, in the part where the author shows that the map
\[
\begin{array}{cccc}
\Phi':&\ClImag{L/K}
&\longrightarrow
& \QFSet \vspace{3mm}\\
&\left[\alpha, \beta\right] & \longmapsto &
\frac{\alpha\cjg{\alpha}x^2-(\cjg{\alpha}\beta+\alpha\cjg{\beta})xy+\beta\cjg{\beta}y^2}{\frac{\cjg{\alpha}\beta-\alpha\cjg{\beta}}{\Omega-\overline{\Omega}}}   
\end{array}
\] 
is well-defined with respect to taking the quotient of $\Cl{L}$ by $\Gr{L/K}$. The author only checks that $\Phi'(I)=\Phi'(\overline{I})$ for any $I\in\ISet{L}$; that is true, but not sufficient. It would have to be proved that $\Phi'(I)=\Phi'(IJ(\overline{J})^{-1})$ for any $I,J\in\ISet{L}$. But as we have seen in Examples~\ref{Ex:3elClL} and~\ref{Ex:4elClL}, there does not exist any factorgroup or subgroup of $\Cl{L}$ which would be of the same size as $\QFSet$. This means in particular that $\Phi'$ is not well-defined.

\section*{Acknowledgment}
 I wish to express my thanks to V\'{\i}t\v{e}zslav Kala for his suggestions and to Jakub Kr\'{a}sensk\'{y} for his help with computations in Appendix~\ref{Sec:App}. \\


\end{document}